\documentclass[11pt]{article}

\usepackage{epsfig,amsmath,amsfonts,amssymb,latexsym,color,amsthm,hhline,multirow,subfigure,graphicx,bm}
\usepackage{enumerate}

\usepackage{tikz}
\usepackage{pgfplots}
\usetikzlibrary{hobby}

\theoremstyle{plain}
\newtheorem{theorem}{Theorem}[section]
\theoremstyle{definition}

\theoremstyle{definition}

\theoremstyle{definition}

\theoremstyle{definition}

\theoremstyle{plain}

\theoremstyle{plain}
\newtheorem{proposition}[theorem]{Proposition}
\theoremstyle{plain}
\newtheorem{lemma}[theorem]{Lemma}
\theoremstyle{plain}

\theoremstyle{definition}

\newcommand{\E}{{\mathbb E}}
\newcommand{\RR}{{\mathbb R}}

\newcommand{\PP}{{\mathbb P}}
\newcommand{\sss}{\scriptscriptstyle}

\newcommand{\vep}{\varepsilon}

\newcommand{\Var}{{\textup{Var}}}
\newcommand{\Cov}{{\textup{Cov}}}
\newcommand{\Corr}{{\textup{Corr}}}

\begin{document}


\title{Mean field stable matchings}

\author{Daniel Ahlberg\thanks{Department of Mathematics, Stockholm University; {\tt \{daniel.ahlberg\}\{mia\}\{matteo.sfragara\}@math.su.se}}\and Maria Deijfen\footnotemark[1] \and Matteo Sfragara\footnotemark[1]}

\date{June 2024}
\maketitle

\begin{abstract}
\noindent Consider the complete bipartite graph on $n+n$ vertices where the edges are equipped with i.i.d.\ exponential costs. A matching of the vertices is stable if it does not contain any pair of vertices where the connecting edge is cheaper than both matching costs. There exists a unique stable matching obtained by iteratively pairing vertices with small edge costs. We show that the total cost $C_{n,n}$ of this matching is of order $\log n$ with bounded variance, and that $C_{n,n}-\log n$ converges to a Gumbel distribution. We also show that the typical cost of an edge in the matching is of order $1/n$, with an explicit density on this scale, and analyze the rank of a typical edge.  These results parallel those of Aldous for the minimal cost matching in the same setting. We then consider the sensitivity of the matching and the matching cost to perturbations of the underlying edge costs. The matching itself is shown to be robust in the sense that two matchings based on largely identical edge costs will have a substantial overlap. The matching cost however is shown to be noise sensitive, as a result of the fact that the most expensive edges will with high probability be replaced after resampling. Our proofs also apply to the complete (unipartite) graph and the results in this case are qualitatively similar.
\vspace{0.3cm}


\vspace{0.2cm}

\end{abstract}


\section{Introduction}
\label{sec:intro}

Consider a situation where a number of objects acting to maximize their own satisfaction are to be matched. Each object ranks the other objects and a matching is then said to be stable if there is no pair of objects that would prefer to be matched to each other rather than their current partners. The concept was introduced in the seminal paper \cite{GS} by David Gale and Lloyd Shapley in 1962 and has since received a lot of attention in many different research areas. In 2012, Lloyd Shapley and Alvin Roth received the Nobel Memorial Prize in Economic Sciences for their work on developing mathematical theory for stable matchings and for applications in economics, respectively. 

The most basic situation described in \cite{GS} consists of matching $n$ men and $n$ women on the marriage market, with only matchings between men and women allowed. This is referred to as the stable marriage problem. It is shown that this problem always (that is, for all ranking lists) has at least one solution, and an algorithm for producing a stable matching is also given. The corresponding problem without the bipartite structure is known as the stable roommates problem, alluding to the problem of allocating a number of students to double rooms in a dormitory. In this case, a stable matching may not exist. A polynomial time algorithm that determines if a matching exists and, if so, outputs the matching is described in \cite{I85}. For more extensive accounts on general theory for stable matchings, we refer to the books \cite{GI89,K96,M13} and references therein.

We will consider stable matchings on the complete bipartite graph $K_{n,n}$ and on the complete graph $K_n$, where the preferences are governed by i.i.d.\ random edge costs. Let us first focus on $K_{n,n}$, which consists of two disjoint vertex sets $V_n=\{v_1,\ldots,v_n\}$ and $V'_n=\{v'_1,\ldots,v'_n\}$, and edge set $E_n=\{(v,v'):v\in V_n,v'\in V'_n\}$. Each edge $e=(v,v')$ in the graph is independently assigned an exponential random variable $\omega(e)$ with mean 1. A \textbf{matching} is a subset $M\subset E_n$ of non-adjacent edges, and a vertex is \textbf{matched} in $M$ if it is contained in an edge of $M$. The matching is \textbf{perfect} if all vertices are matched. The \textbf{partner} of $v$ in $M$ is given by
$$
M(v)=\left\{
        \begin{array}{ll}
        v' & \mbox{if } (v,v')\in M;\\
	\emptyset & \mbox{if $v$ is not matched},
        \end{array}
            \right.
$$
and the \textbf{matching cost} of $v$ in $M$ is defined as
$$
c(v)=\left\{
        \begin{array}{ll}
        \omega((v,M(v))) & \mbox{if $v$ is matched};\\
	\infty & \mbox{if $v$ is not matched}.
        \end{array}
            \right.
$$
A matching is \textbf{stable} if there do not exist any pair of vertices with an edge between them that is cheaper than both matching costs, that is, if
\begin{equation}\label{eq:stable}
\forall v\in V_n, v'\in V'_n: (v,v')\not\in M\Rightarrow \omega((v,v'))>\min\{c(v),c(v')\}.
\end{equation}
Vertices hence rank potential partners based on the cost of the connecting edge, and prefer to be matched as cheaply as possible. A vertex pair violating \eqref{eq:stable} consists of vertices that would prefer to be matched to each other rather than to their current partners, and is therefore called an \textbf{unstable pair}. The following algorithm yields an almost surely unique stable matching on $K_{n,n}$:\medskip

\noindent  \textbf{Greedy algorithm.} First select the cheapest edge $(v,v')$ in the graph and include this in the matching. Erase all other edges incident to $v$ and $v'$. Then select the cheapest edge $(u,u')$ among the remaining edges and include this in the matching. Erase all other edges incident to $u$ and $u'$. Repeat until all vertices have been matched. \medskip

It follows by induction over the steps in the algorithm that all edges created by the algorithm must be included in any stable matching, since omitting any of the edges would result in an unstable pair. Note that it is important that the edge costs are almost surely distinct. Also note that the matching is perfect. 

The concept of a stable matching can be defined analogously on the complete graph $K_n$, and the algorithm then produces a unique stable matching which is perfect if and only if $n$ is even (and otherwise has exactly one unmatched vertex). In our setting, a stable matching hence always exists also in the non-bipartite case. This is because basing the preferences on random edge cost leads to heavily correlated ranking lists. Indeed, if $v$ is highly ranked by $v'$, it means that the edge $(v,v')$ has a small cost, which implies that $v'$ is most likely also highly ranked by $v$.

Matchings on weighted graphs have previously been studied in connection with the so-called random assignment problem. The task is then to assign $n$ jobs to $n$ machines in such a way that the total cost of performing all jobs is minimized. The input consists of a complete bipartite graph with i.i.d.\ exponential edge weights, specifying the pairwise costs, and the goal is to find a perfect matching that minimizes the total cost
$$
C(M)=\sum_{v\in V}c(v).
$$
In the seminal paper \cite{aldous01}, Aldous proved that the total cost of the minimal matching converges to $\pi^2/6$, which had been conjectured for quite some time. He also analyzed the cost and rank of a typical edge in the minimal matching, and showed that any matching differing from the minimal one in $O(1)$ edges is asymptotically significantly more costly; see Section \ref{sec:minimal} for further details. Background and results predating \cite{aldous01} can be found in \cite{mezpar87,par98,steele97}, and later results e.g.\ in \cite{wastlund09,wastlund12}. 

In this paper, we derive results for the stable matching that parallel those of Aldous \cite{aldous01} for the minimum matching; see Theorems \ref{theorem1}-\ref{theorem4} below. The behaviour that we encounter differs from that of the minimum matching in that the greedy matching results in a heavier edges being added at the end of the process. We then proceed to study the sensitivity of the stable matching with respect to small perturbations of the edge costs. In analogy with Aldous’ asymptotic essential uniqueness (AEU) property, we show that updating a small proportion of the edge costs has a limited effect on which edges are contained in the matching (Theorem \ref{theorem4}). As highlight of the paper, however, we show that the most expensive edges (the `tail') of the matching are very likely to be replaced by such a perturbation (Theorem \ref{thm:tailsensitivity}). This is a consequence of the larger cost of the stable matching compared to the minimum matching, where the same behaviour should not occur. Moreover, although the bulk of the stable matching contributes with the lion part of its cost, most of the randomness in its total cost comes from its tail. As a consequence of the sensitivity of the most expensive edges in the matching it follows that the matching cost is highly sensitive to resampling a small proportion of the edge weights (Theorem \ref{thm:weightsensitivity}). To the best of our knowledge, this is the first confirmed instance where chaotic behaviour of the minimising structure and noise sensitivity of the minimised function does not come hand in hand; however, see recent work of Israeli and Peled \cite{IP} for results of a similar flavour.

\subsection{Results}
\label{sec:results}

Let $S_{n,n}$ denote the unique stable matching on $K_{n,n}$ based on i.i.d.\ exponential edge weights $\{\omega(e)\}_{e\in E_n}$ with mean 1, and write $C_{n,n}=C(S_{n,n})$ for the total cost of the matching. Our first result specifies the asymptotic behavior of $C_{n,n}$ and is the analogue of \cite[Theorem 1]{aldous01}. In contrast to \cite[Theorem 1]{aldous01}, we also obtain a distributional limit of the centered total cost.

\begin{theorem}[The total cost]
\label{theorem1}
We have that 
$$
\lim_{n \to \infty} \frac{\E[C_{n,n}]}{\log n} = 1 \qquad \text{ and } \qquad \lim_{n \to \infty} \Var(C_{n,n}) =\frac{\pi^2}{6}.
$$
Furthermore, $C_{n,n}-\log n\stackrel{d}{\to} G$, where $G$ is a Gumbel distributed random variable.
\end{theorem}

There are $n^2$ edges in $K_{n,n}$ and hence the cost of the cheapest edge, which is for sure part of the stable matching, is of the order $1/n^2$. The typical cost of an edge in the matching however is of the order $1/n$, as stated in the next theorem. Note that the vertices are exchangeable and hence the matching cost $c(v)$ of vertex $v$ has the same distribution for all vertices $v\in V_n$. This is also the distribution of the cost of a randomly chosen edge contained in the matching. Scaling the cost by $n$ turns out to give rise to a proper random variable with an explicit distribution in the limit. This is the analogue of \cite[Theorem 2]{aldous01}.

\begin{theorem}[The typical matching cost]\label{theorem2}
For any vertex $v$, the cost $nc(v)$ in $S_{n,n}$ converges in distribution as $n\to\infty$ to a random variable $W$ with density
\begin{equation}\label{typdens}
f_W(x) = \frac{1}{(1+x)^2}, \qquad x \in [0, \infty).
\end{equation}
\end{theorem}

Next consider the typical rank of an edge in the matching. Specifically, order the edges incident to vertex $v\in V_n$ in $K_{n,n}$ according to increasing edge cost and let $R_n$ be a random variable indicating the rank of the edge that is used in the stable matching, that is, $R_n=r$ if the matching uses the $r$th cheapest edge of vertex $v$. The following result is the analogue of \cite[Theorem 3]{aldous01}.

\begin{theorem}[The edge rank]
\label{theorem3}
We have that $R_n\stackrel{d}{\to} R$ as $n\to \infty$, where
\begin{itemize} 
\item[\rm{(i)}] $\PP(R=1) = e \int_1^{\infty} \frac{e^{-x}}{x}\, dx \approx 0.596;$
\item[\rm{(ii)}] $\PP(R \geq r) \sim \frac{1}{r}$, as $r \to \infty$.
\end{itemize}
\end{theorem}

Some structures arising from i.i.d.\ configurations have recently been shown to exhibit a chaotic behavior with respect to perturbations of the underlying configuration. This direction of research first arose in the literature on disordered systems, to which combinatorial optimization problems such as minimal matchings are considered related. Specifically, it has been observed that resampling only a very small fraction of the underlying configuration can cause substantial changes to some structures; see e.g.\ \cite{ADS22,BLZ20,chatterjee14a,GH20a}. Our next result shows that this is not the case for $S_{n,n}$. Let $\omega=\{\omega(e)\}_{e\in E_n}$ and $\omega'=\{\omega'(e)\}_{e \in E_n}$ be two independent random configurations of i.i.d.\ mean 1 exponential edge costs, and let $\{U(e)\}_{e\in E_n}$ be i.i.d.\ uniform variables on $[0,1]$ independent of $\omega$ and $\omega'$. For $\varepsilon \in [0,1]$, define $\omega_\vep=\{\omega_\vep(e)\}_{e\in E_n}$ to be a configuration where a fraction $\vep$ of the entries in $\omega$ are replaced by their counterparts in $\omega'$, that is, 
\begin{equation}\label{eq:omega_t}
\omega_{\varepsilon}(e):=\left\{
\begin{aligned}
\omega(e) &&& \text{if }U(e)>\varepsilon,\\
\omega'(e) &&& \text{if }U(e)\le \varepsilon.
\end{aligned}
\right.
\end{equation}
Let $S_{n,n}^{\varepsilon}$ denote the stable matching based on $\omega_\vep$. The following result shows that the fraction of edges in $S_{n,n}^0$ that are also part of $S_{n,n}^\vep$ converges to 1 as $\vep\to 0$.

\begin{theorem}[Robustness of the matching]\label{theorem4}
For any $\varepsilon > 0$, there exists a constant $C >7$ such that
$$
\lim_{n \to \infty} \frac{\E\left[|S_{n,n}^0 \cap S_{n,n}^{\varepsilon}|\right]}{n} \geq 1 - C \frac{1}{\log\left(\frac{1}{\varepsilon}\right)}.
$$
\end{theorem}

While a small perturbation of the edge costs will leave the stable matching largely intact, it turns out that the most expensive edges of the matching, on the contrary, will be replaced with high probability. For $m\ge1$ and $\vep\in [0,1]$, let $L_\vep(m)$ denote the the sets of vertices corresponding to the $m$ most expensive edges in the matching $S_{n,n}^\vep$ (that is, the last $m$ edges to be picked by the greedy algorithm).

\begin{theorem}[Sensitivity of the tail]\label{thm:tailsensitivity}
Let $m\ge1$ and $\vep\in(0,1]$ satisfy $m\ll\vep\log n$ as $n\to\infty$. Then, with high probability as $n \to \infty$, none of the edges in $L_0(m)$ remain in the matching after perturbation and the two sets $L_0(m)$ and $L_\vep(m)$ are hence disjoint.
\end{theorem}


Let $C_{n,n}^{\varepsilon}$ denote the total cost of the stable matching based on $\omega_\vep$. It will turn out that the most expensive edges are responsible for most of the randomness in the matching cost. A consequence of the above result is hence that the total cost of the matching is sensitive to the perturbation of the edge costs, in the sense that the matching costs before and after resampling are asymptotically uncorrelated. 

\begin{theorem}[Noise sensitivity of the matching cost]\label{thm:weightsensitivity}
For $\vep\log n\gg1$ we have that 
$$
\Corr\big(C_{n,n}^0,C_{n,n}^\vep\big)\to0\mbox{ as }n\to\infty.
$$
\end{theorem}

The study of noise sensitivity was initiated by Benjamini, Kalai and Schramm \cite{BKS} in the context of Boolean functions. The topic has since developed substantially, but results are still mainly restricted to Boolean functions. Theorem \ref{thm:weightsensitivity} is one of the first instances of noise sensitivity for a more general function (the matching cost). 

\subsection{Comparison with the minimal matching}\label{sec:minimal}

As mentioned above, the asymptotic total cost of the minimal matching on $K_{n,n}$ is a constant $\pi^2/6$, while for the stable matching it grows logarithmically with $n$ according to Theorem \ref{theorem1}. Indeed, the stable matching arises from a greedy algorithm that selects cheap edges in the early stages, but will pay a price for this in the later stages when more expensive edges have to be selected. The typical cost of an edge in the matching however is of the order $1/n$ in both matchings. For the stable matching, the density of the limiting typical edge cost $W$ on this scale is given by \eqref{typdens}, and for the minimal matching it is shown in \cite[Theorem 2]{aldous01} to equal
$$
h(x) = \frac{e^{-x}(e^{-x}-1+x)}{(1-e^{-x})^2}, \quad x \geq 0. 
$$
The distribution has an exponentially decaying tail for the minimal matching and a power law tail with infinite mean for the stable matching indicating that, also on the typical scale, the stable matching is more likely to produce edges with a large cost. 

At the other end of the spectrum, the expected total number of edges in $K_{n,n}$ with cost at most $x/n$ is given by $n^2\PP(\omega(e)\leq x/n)\sim xn$ for small $x$, and the expected number of edges in the matching with cost at most $x/n$ is given by $n\PP(W\leq x)$, which according to the given densities of $W$ scales as $xn$ for the stable matching and as $xn/2$ for the minimal matching for small $x$. The fraction of edges with a small weight on the typical scale that will be a part of the matching hence equals 1 for the stable matching and 1/2 for the minimal matching, so that the stable matching hence includes all but a vanishing fraction of the cheap edges on the typical scale, while the minimal matching uses only half of those edges. Being less greedy in this regime turns out to be beneficial for the minimal matching, since it helps to avoid the expensive edges created at the end of the algorithm by the stable matching.

The edge rank $R_n$ is shown in \cite[Theorem 3]{aldous01} to converge to a random variable $R$ with probability  function  $\PP(R=r)=2^{-r}$. In particular, the probability that the cheapest edge of a vertex is used is 1/2. In the stable matching, on the other hand, this probability is approximately 0.596 and the rank distribution has a power law tail with infinite mean. This again reflects the fact that the stable matching is more likely to use the very cheapest edges, but will in return include more edges with a large cost.

As for the robustness of the stable matching established in Theorem \ref{theorem4}, a related property, referred to as an asymptotic essential uniqueness (AEU) property, is established for the minimal matching in \cite[Theorem 4]{aldous01}. It is shown that, if a matching differs from the minimal one by a proportion at least $\delta$, then its cost is at least $\vep=\vep(\delta)$ larger than the minimal one. The minimal matching is hence unique in the sense that a matching with a cost close to the optimal one must to a large extent coincide with the minimal matching. 

\subsection{Outline of proofs}\label{ss:proof}

Write $Y_k$ for the cost of the edge selected in the $k$th step of the greedy algorithm. In the first step, the cost is the minimum of $n^2$ exponential variables with mean 1 and is hence Exp$(n^2)$-distributed. With the convention that $Y_0=0$, by the memoryless property of the exponential distribution, we can for $k\ge1$ write
\begin{equation}\label{yk}
Y_k=Y_{k-1}+X_k,
\end{equation}
where $X_k$ is the minimum of $(n-k+1)^2$ exponential variables with mean 1 and hence Exp$((n-k+1)^2)$-distributed. The total cost is obtained as
\begin{equation}\label{tk}
C_{n,n}=\sum_{k=1}^nY_k=\sum_{k=1}^n(n-k+1)X_k=\sum_{k=1}^nZ_k\quad\mbox{where }Z_k\sim \mbox{Exp}(k).
\end{equation}
Theorem \ref{theorem1} follows immediately from this expression, and Theorem \ref{theorem2} follows by analyzing the weight of the $U$th selected edge, where $U$ is uniform on $[n]=\{1,2,\ldots ,n\}$. 

Theorem \ref{theorem3} is proved by transferring the problem to a limiting object known as the Poisson Weighted Infinite Tree (PWIT), which is also the strategy used in \cite{aldous01} (there also the first two results are obtained from computations on the PWIT, since the algorithm to obtain the minimal matching is less explicit). As a preparation for this, we extend the concept of stable matchings to general (possibly infinite) graphs and adapt the greedy algorithm. We also explore connections between the stable matching and so-called descending paths, which are paths with strictly decreasing edge costs. These ideas have previously appeared in \cite{holmarper20}. To obtain the PWIT as a local limit of the weighted version of $K_{n,n}$, we transform the edge costs to the typical scale by multiplying them by $n$. We then show that the rank on $K_{n,n}$ converges in distribution to the rank on the PWIT, where the latter can be explicitly computed. 

Theorem \ref{theorem4} is proved by making use of the relation between the stable matching and descending paths. Specifically, changes in the stable matching arising from resampling a certain proportion of the edge costs can be estimated by aid of crude bounds on the set of descending paths emanating from a given vertex.

The results concerning sensitivity of the most expensive edges and the total matching cost are established by splitting the vetex set into two sets $L_0(m)$ and $L^c_0(m)$, corresponding to the $m$ most expensive and the $n-m$ cheapest edges of the original matching, respectively. Theorem \ref{thm:tailsensitivity} is proved by showing that, after resampling, every vertex in $L_0(m)$ is desired by many vertices in $L_0^c(m)$ and, with high probabiliy, the desire is reciprocated. This implies that the vertices in $L_0(m)$ are with high probabiliy matched to vertices in $L_0^c(m)$ after resampling. As for Theorem \ref{thm:weightsensitivity}, we observe that most of the matching cost is generated by the bulk of the matching, which turns out to be essentially deterministic, while most of the randomness comes from the last few, most expensive, edges. We then construct the original matching $S_{n,n}$ and the matching $S_{n,n}^{\vep}$ based on the perturbed configuration dynamically by adding edges at times prescribed by their costs. Most edges are the same in both matchings but, by Theorem~\ref{thm:tailsensitivity}, the last edges correspond to disjoint subgraphs and are therefore generated by independent times/costs. Sine this phase is responsible for most of the randomness in the matching, the correlation of the matching costs will be small.

\subsection{Results for the complete graph}\label{sec:complete}

Before proceeding with the proofs, we comment briefly on results for the stable matching on the complete graph $K_n$, where $n$ is assumed to be even. All our proofs extend, with very minor adjustments, to this case. For the total weight $C_n$ we obtain that 
$$
\frac{\E [C_n]}{\log n}\to 1/2 \,\, \mbox{  and  }	\,\mbox{Var}(C_n)\to \pi^2/8.
$$
The number of edges in the matching is $n/2$ on $K_n$ while it is $n$ on $K_{n,n}$, so the expected total matching cost is asymptotically the same in relation to the number of edges. This is proved by noting that, on $K_n$, the representation in \eqref{tk} is replaced by
\begin{equation}\label{eq:Kntot}
C_n=\sum_{k=1}^{n/2}(n/2-k+1)X'_k=\sum_{k=1}^{n/2}Z'_k
\end{equation}
where $X'_k\sim \mbox{Exp}\left(\binom{n-2k+2}{2}\right)$ and $Z'_k\sim \mbox{Exp}(2k-1)$. The centered total matching cost $C_n-\log n/2$ converges in distribution to a proper random variable also on $K_n$. However, perhaps somewhat surprisingly, the limiting distribution is not a Gumbel. We explain this in more detail after the proof of Theorem \ref{theorem1}. Our other results apply in identical formulations also on $K_n$, except that the normalization in Theorem \ref{theorem4} is $n/2$ (the number of edges) instead of $n$. As for Theorem \ref{theorem2}, the proof is identical, except that we need to work with $X_k'$ instead of $X_k$ and recall that there are $n/2$ instead of $n$ edges to choose from. Theorem \ref{theorem3} is proved by computing the rank on the limiting PWIT. It is well known that also the weighted graph $K_n$ converges locally to the PWIT (see Section \ref{sec:PWIT}) and the distribution of the rank is therefore the same. The proof of Theorem \ref{theorem4} applies verbatim on $K_n$, and so do the proofs of Theorem \ref{thm:tailsensitivity} and \ref{thm:weightsensitivity}, provided we again work with $X_k'$ instead of $X_k$ and recall that there are $n/2$ edges in total.

\subsection{Further work}

One natural question is to what extent our results generalize to other distributions of the edge costs. Some results will certainly be different, for instance the quantification of the matching cost and its fluctuations in Theorem \ref{theorem1} will be affected, as well as the explicit density of the typical matching cost in Theorem \ref{theorem2}. Note however that the stable matching is defined only through the relative ordering of the edge costs. This implies that, if the edge costs are transformed by a strictly increasing continuous function, then the stable matching does not change. Transforming the costs by a strictly decreasing function, on the other hand, yields a stable matching where expensive edges in the original configuration are preferred. The edges can then be relabelled by inverting their order, so that in particular the rank $R_n$ has the same meaning as above. Since Theorem \ref{theorem3} is about the relative ordering of the edges, it can hence be extended to all continuous cost distributions on $(0, \infty)$. Similarly, Theorem \ref{theorem4} is only concerned with the matching as a geometrical object and thus also extends to all continuous cost distributions on $(0, \infty)$. The proofs of Theorem \ref{thm:tailsensitivity} and Theorem \ref{thm:weightsensitivity} rely heavily on specific estimates for the exponential distribution and the memoryless property so would need to be revised for other distributions.

Another question is whether the decorrelation in Theorem \ref{thm:weightsensitivity} ceases to hold when instead $\vep \log n\ll 1$. We conjecture that this is indeed the case, so that there is hence a transition at $\vep\sim (\log n)^{-1}$: For $\vep\log n\gg 1$, the matching costs decorrelate, while for $\vep \log n \ll 1$ they do not.

\section{Proofs of Theorems \ref{theorem1} and \ref{theorem2}}

In this section we give the short proofs of Theorem \ref{theorem1} and Theorem \ref{theorem2}.

\begin{proof}[Proof of Theorem \ref{theorem1}]
Recall the expression \eqref{tk} for the total cost and note that the variables $\{Z_k\}_{k\geq 1}$ are independent. The expectation is given by
$$
\E[C_{n,n}]=\sum_{k=1}^n\E[Z_k]=\sum_{k=1}^n\frac{1}{k}\sim \log n
$$
and the variance by
$$
\mbox{Var}(C_{n,n})=\sum_{k=1}^n\mbox{Var}(Z_k)=\sum_{k=1}^n\frac{1}{k^2}\to \frac{\pi^2}{6}.
$$
To obtain the distributional limit, define $\widetilde{Z}_k=Z_k-\E[Z_k]=Z_k-1/k$ and $\widetilde{C}_{n,n}=\sum_{k=1}^n \widetilde{Z}_k$. The moment generating function of $\widetilde{C}_{n,n}$ is given by
$$
\Psi_{\tilde{C}_{n,n}}(t)=\prod_{k=1}^n\frac{1}{1-t/k}\,e^{-t/k}\to \Gamma(1-t)e^{-\gamma t} \quad \mbox{as }n\to\infty,
$$
with $\gamma\approx 0.577$ denoting the Euler Mascheroni constant, where the convergence follows from the expansion $\Gamma(t)=\frac{e^{\gamma t}}{t}\prod_{k=1}^\infty \frac{1}{1+t/k}e^{t/k}$ and the relation $\Gamma(1-t)=-t\Gamma(-t)$ for the Gamma function. The limit is recognized as the generating function of a Gumbel variable with location parameter $-\gamma$ and scale parameter $1$. Finally, note that $\widetilde{C}_{n,n}=C_{n,n}-\sum_{k=1}^n1/k$, where $\sum_{k=1}^n1/k\sim \log n$.
\end{proof}

Before proceeding with the proof of Theorem \ref{theorem2}, we comment on the distributional limit for the stable matching on $K_n$. The total cost $C_n$ is then given by \eqref{eq:Kntot}. Centering $Z_k'$, as in the proof of Theorem \ref{theorem1}, we obtain $\widetilde{Z}_k'$ and the corresponding sum $\widetilde{C}_n$ with moment generating function
$$
\Psi_{\tilde{C}_n}(t)=\prod_{k=1}^{n/2}\frac{1}{1-t/(2k-1)}\,e^{-t/(2k-1)}=
\prod_{k=1}^{n}\frac{1}{1-t/k}\,e^{-t/k} \left[\prod_{k=1}^{n/2}\frac{1}{1-t/2k}\,e^{-t/2k}\right]^{-1}.
$$
Using the same results for the Gamma function as in the proof of Theorem \ref{theorem1}, we obtain that the first product on the right-hand side converges to $\Gamma(1-t)e^{-\gamma t}$ while the second product converges to $\Gamma(1-t/2)e^{-\gamma t/2}$. Hence
$$
\Psi_{\tilde{C}_n}(t)\to \frac{\Gamma(1-t)}{\Gamma(1-t/2)}e^{\gamma t/2}\quad \mbox{as }n\to\infty.
$$
We conclude, as in the proof of Theorem \ref{theorem1}, that $C_n-\log n /2$ converges in distribution to a random variable with this generating function. However, the generating function does not correspond to a Gumbel distribution, and hence the limiting distribution is not Gumbel.

\begin{proof}[Proof of Theorem \ref{theorem2}]
Recall from Section \ref{ss:proof} that the cost of the edge selected in the $k$th step is given by $Y_k=\sum_{i=1}^kX_k$, where $X_k\sim$Exp$((n-k+1)^2)$. Let $U$ be uniform on $[0,1]$. Note that the matching cost $c(v)$ of vertex $v$ has the same distribution as the cost $Y_{\lceil Un\rceil}$ of a randomly chosen edge in the matching. To analyze the latter, note that, for $\alpha\in[0,1)$, we have that
$$
\E[Y_{\lceil \alpha n\rceil}]=\sum_{i=1}^{\lceil \alpha n\rceil}\frac{1}{(n-i+1)^2}\sim \int_1^{\alpha n}\frac{1}{(n-x)^2}dx\sim \frac{\alpha}{n(1-\alpha)}
$$
and 
$$
\mbox{Var}[Y_{\lceil \alpha n\rceil})=\sum_{i=1}^{\lceil \alpha n\rceil}\frac{1}{(n-i+1)^2}\sim \int_1^{\alpha n}\frac{1}{(n-x)^4}dx=O(1/n^3).
$$
Hence $nY_{\lfloor \alpha n\rfloor}$ converges in probability to $\alpha/(1-\alpha)$. Now fix $\vep>0$ and decompose
\begin{equation*}
\begin{array}{lll}
\PP\left(nY_{\lceil Un\rceil}\leq \frac{\alpha}{1-\alpha}\right) & = & \PP\left(nY_{\lceil Un\rceil}\leq \frac{\alpha}{1-\alpha} , U\leq \alpha-\vep\right)+ 
\PP\left(nY_{\lceil Un\rceil}\leq \frac{\alpha}{1-\alpha} , U\geq \alpha+\vep\right)\\
& &+\PP\left(nY_{\lceil Un\rceil}\leq \frac{\alpha}{1-\alpha} , U\in(\alpha-\vep,\alpha+\vep)\right).
\end{array}
\end{equation*}
The first term converges to $\PP(U\leq \alpha-\vep)=\alpha-\vep$, the second term converges to 0 and the last term is bounded from above by $\PP(U\in(\alpha-\vep,\alpha+\vep))=2\vep$. Sending $n\to\infty$ and $\vep\to 0$ yields that the limit equals $\alpha$. Hence $nY_{\lfloor Un\rfloor}$ converges to a random variable $W$ with a distribution function satisfying $F_W(\frac{\alpha}{1-\alpha})=\alpha$. The latter can be inverted to $F_W(x) = \frac{x}{1+x}$, which corresponds to the stated density.
\end{proof}

\section{Stable matchings, descending paths and the PWIT}

In this section we extend the definition of stable matchings to general weighted graphs, introduce the notion of descending paths and describe how stable matchings are related to such paths.  We then define the PWIT, which is a well-known infinite tree arising as local limit of $K_{n,n}$ with exponential weights. This will be useful in the next section, where Theorem \ref{theorem3} is proved by transferring the computations to the PWIT and Theorem \ref{theorem4} by exploiting the connection between the stable matching and descending paths. Some of these auxiliary results can be found in similar form in \cite{holmarper20}, and we present them here for completeness.

\subsection{Stable matchings and descending paths}

Consider a weighted graph $G=(V,E)$, with finite or countably infinite vertex set $V$, edge set $E$ and edge costs $\{\tau(e)\}_{e\in E}$ (random or deterministic).  A matching on $G$ is a subset $M\subset E$ of non-adjacent edges. The concepts of matched vertices, perfect matching, the partner $M(v)$ and matching cost $c(v)$ of a vertex $v$ are defined analogously as in Section \ref{sec:intro}. A matching is stable if
$$
\forall u,v\in V \mbox{ with }(u,v)\in E: (u,v)\not\in M\Rightarrow \tau(u,v)>\min\{c(u),c(v)\}.
$$
Note that, if $u$ and $v$ are neighbors in $G$ and $\tau(u,v)<\infty$, then $u$ and $v$ cannot both be unmatched in a stable matching. A stable matching may not exist and, if it does, it may not be unique. Sufficient conditions for existence and uniqueness involve the concept of descending paths. For a weighted graph $G=(V,E)$, a descending path is a weighted subgraph consisting of a sequence of adjacent edges $e_1, e_2,\ldots$ such that $\tau(e_1)>\tau(e_2)>\tau(e_3)>\ldots$.  The set of descending paths emanating from a given vertex $v\in V$ is denoted by $D_v(G)$.

The following proposition from \cite{holmarper20} gives conditions that guarantee the existence of a unique stable matching. We include a proof for completeness.

\begin{proposition}[Holroyd, Martin, Peres (2020)]\label{prop:existsunique}
Given a weighted graph $G = (V,E)$, there exists a unique stable matching $S(G)$ if
\begin{itemize}
\item[\rm{(i)}] the edge costs are finite and all distinct;
\item[\rm{(ii)}] for each vertex $v \in V$ and all finite $s>0$, the set of vertices connected to $v$ by an edge with weight less than $s$ is finite;
\item[\rm{(iii)}] there are no infinite descending paths. 
\end{itemize}
\end{proposition}

\begin{proof}
We prove the proposition by giving an algorithm that produces the matching. The greedy algorithm described in Section \ref{sec:intro} works only on finite graphs, but the following algorithm is well-defined on any graph satisfying (i) and (ii):\medskip

\noindent \textbf{General greedy algorithm.} Two vertices $u$ and $v$ are called potential partners if $(u,v)\in E$, and two potential partners $u$ and $v$ are called mutual favourites if $(u,v)$ is the cheapest among all edges of $u$ and also the cheapest among all edges of $v$. Note that (i) and (ii) guarantee that any vertex has a unique cheapest edge. Match all mutual favourites and remove them from the graph. Then match all mutual favorites in the remaining graph. Repeat (possibly indefinitely) until no unmatched potential partners remain.\medskip

We claim that this produces a unique stable matching. As for the algorithm in Section \ref{sec:intro}, it follows from induction over the stages in the algorithm that all edges created must be included in any stable matching, since otherwise there would be an unstable pair. We also need to show that all vertices that are left unmatched by the algorithm are unmatched in any stable matching. To this end, let $v$ be a vertex that is unmatched in the matching $S$ arising from the algorithm, and assume there is another stable matching $S'$ where $v$ is matched, say to $u_1$. The fact that $v$ is not matched to $u_1$ in $S$ means that $u_1$ must be matched to a vertex $u_2$ with $\tau(u_1,u_2)<\tau(v,u_1)$ in $S$, since $v$ and $u_1$ would otherwise constitute an unstable pair in $S$. Similarly, the fact that $u_1$ is not matched to $u_2$ in $S'$ means that $u_2$ must be matched to a vertex $u_3$ with $\tau(u_2,u_3)<\tau(u_1,u_2)$, since $u_1$ and $u_2$ would otherwise constitute an unstable pair in $S'$. Iterating this leads to the conclusion that the graph must contain an infinite descending path. If no such path exists, there can hence not exist stable matchings where $v$ is matched.
\end{proof}

Descending paths turn out to have further importance for the stable matching. In essence, in order to find out if a vertex is matched in the stable matching and, if so, identify its partner, it is sufficient to investigate the set of descending paths emanating from the vertex. To formulate this, write $S(G)$ for the unique stable matching of a graph satisfying the assumptions of Proposition \ref{prop:existsunique}. Also, denote the set of descending paths including only edges with weight at most $s>0$ by $D_v(G,s)$.

\begin{proposition}\label{prop:stabledescending}
Consider a weighted graph $G = (V,E)$ satisfying conditions {\rm{(i)-(iii)}} of Proposition \ref{prop:existsunique} and fix a vertex $v \in V$. For any $s>0$, we have that 
\begin{equation}\label{s_inc}
S(D_v(G,s))\subseteq S(D_v(G))\subseteq S(G).
\end{equation}
Furthermore, if $v$ is unmatched in $S(D_v(G,s))$ for all $s>0$, then $v$ is unmatched in $S(G)$.
\end{proposition}


\begin{proof}
Note that $D_v(G,s)\subseteq D_v(G)$. Assume that $S(D_v(G,s))\not\subseteq S(D_v(G))$. This means that there exists a vertex $v\in D_v(G,s)$ that is matched to a vertex $u_2$ in $S(D_v(G,s))$, but that is matched to another vertex $u_1$ (or possibly unmatched) in $S(D_v(G))$. We consider the case when $\tau(v,u_2)<\tau(v,u_1)$, so that $v$ has a higher matching cost in $S(D_v(G))$ (including also the possibility that $v$ is unmatched in $S(D_v(G))$), but the opposite case can be handled analogously. By definition of $D_v(G,s)$, no vertex that is matched in $S(D_v(G,s))$ prefers a vertex in $D_v(G,s)^c$ before its partner in $S(D_v(G,s))$, since edges to vertices in $D_v(G,s)^c$ are more expensive than edges to vertices in $D_v(G,s)$. It follows that the vertex $u_2$ must be matched in $S(D_v(G))$ to a vertex $u_3\in D_v(G,s)$ that is matched in $S(D_v(G,s))$ and with $\tau(u_2,u_3)<\tau(v,u_2)$, since $v$ and $u_2$ would otherwise constitute an unstable pair in $S(D_v(G))$. Let $u_4$ denote the partner of $u_3$ in $S(D_v(G,s))$. Then $\tau(u_3,u_4)<\tau(u_3,u_2)$, since otherwise $u_3$ and $u_4$ would be unstable in $S(D_v(G,s))$. Furthermore, as with $u_2$, the vertex $u_4$ must be matched in $S(D_v(G))$ to a vertex $u_5\in D_v(G,s)$ that is matched in $S(D_v(G,s))$ and with $\tau(u_4,u_5)<\tau(u_3,u_4)$, since $u_3$ and $u_4$ would otherwise constitute an unstable pair in $S(D_v(G))$. Iterating this leads to an infinite descending chain, which by assumption does not exist, and therefore a contradiction. We conclude that all vertices in $D_v(G,s)$ that are matched in $S(D_v(G,s))$ must be matched to the same partner in $S(D_v(G))$, that is, $S(D_v(G,s))\subseteq S(D_v(G))$. The other inclusion in \eqref{s_inc} follows from an analogous argument, noting that no vertex in $D_v(G)$ prefers a vertex in $D_v(G)^c$ before its partner in $S(D_v(G))$.

To show the last statement, assume that $v$ is unmatched in $S(D_v(G,s))$ for all $s>0$, but that $v$ is matched in $S(G)$, say to $u$. The matching cost of $v$ in $S(G)$ is $c(v)=\tau(v,u)$. By \eqref{s_inc}, the vertex $u$ cannot be matched to a different vertex in $S(D_v(G,c(v)))$, since it would then be matched to this other vertex also in $S(G)$. Hence both $u$ and $v$ are unmatched in $S(D_v(G,c(v)))$ and thus constitute an unstable pair. We conclude that $v$ cannot be matched in $S(G)$.
\end{proof}

\subsection{The PWIT}\label{sec:PWIT}

The Poisson Weighted Infinite Tree (PWIT) was first introduced in \cite{aldous92}. To describe it, consider first a root vertex with an infinite number of children. The edges from the root to the children are assigned weights according to a Poisson process with rate 1. Recursively, each child is then given an infinite number of new children and the edges to these new children are again assigned weights according to the arrival times of independent Poisson processes with rate 1. Continuing this procedure, leads to a rooted infinite tree $\mathcal{T}$ known as the PWIT. Formally, the PWIT is a rooted weighted graph with vertex set
$$
\mathcal{V}= \cup_{k=0}^\infty \mathbb{N}^k = \{0, 1,2, \dots, 11, 12, \dots, 21,22, \dots, 111, 112, \dots \},
$$
where $0$ is the root, and edges $(v,vj)$, for each $v \in \mathcal{V}$ and $j \in \mathbb{N}$, where $vj$ is referred to as a child of $v$. For $v \in \mathcal{V}$, let $(T_j^{\sss (v)})_{j \in \mathbb{N}}$ be the points (in increasing order) of a Poisson process on $\mathbb{R}_+$ with rate $1$. The cost of an edge $(v, vj)$ is given by $T_{vj} = T_j^{\sss(v)}$, where we write $T_{0j} = T_j$; see Figure \ref{fig:PWIT1}.

\begin{figure}[htbp]
\begin{center}
\begin{tikzpicture}[scale=1]
\filldraw[black] (0,0) circle (2pt) node{};
\draw (-0.2,-0.1) -- (-3, -3);
\draw (-0.1,-0.1) -- (-1, -3);
\draw (0.0,-0.1) -- (1, -3);
\draw (0.1,-0.1) -- (3, -3);
\draw [dashed] (0.2,-0.1) -- (5, -3);
\filldraw[black] (-3,-3) circle (2pt) node{};
\filldraw[black] (-1,-3) circle (2pt) node{};
\filldraw[black] (1,-3) circle (2pt) node{};
\filldraw[black] (3,-3) circle (2pt) node{};
\filldraw[black] (-5.5,-5) circle (2pt) node{};
\filldraw[black] (-4.5,-5) circle (2pt) node{};
\draw (-3.15,-3.1) -- (-5.5, -5);
\draw (-3.1,-3.1) -- (-4.5, -5);
\draw (-3.05,-3.1) -- (-3.5, -5);
\draw [dashed] (-3.0,-3.1) -- (-2.5, -5);
\filldraw[black] (-1.5,-5) circle (2pt) node{};
\filldraw[black] (-0.5,-5) circle (2pt) node{};
\draw (-1.05,-3.1) -- (-1.5, -5);
\draw (-1,-3.1) -- (-0.5, -5);
\draw (-0.95,-3.1) -- (0.5, -5);
\draw [dashed] (-0.9,-3.1) -- (1.5, -5);
\draw (0,0.4) node[] {$0$};
\draw (-2.1,-1.5) node[] {$T_1$};
\draw (-1.1,-1.9) node[] {$T_2$};
\draw (0.25,-2) node[] {$T_3$};
\draw (-3.45,-2.9) node[] {$1$};
\draw (-4.7,-3.9) node[] {$T_{11}$};
\draw (-1.45,-2.9) node[] {$2$};
\draw (-1.7,-3.9) node[] {$T_{21}$};
\draw (0.6,-2.9) node[] {$3$};
\draw (-5.9,-4.9) node[] {$11$};
\draw (-4.9,-4.9) node[] {$12$};
\draw (-1.9,-4.9) node[] {$21$};
\draw (-0.9,-4.9) node[] {$22$};
\end{tikzpicture}
\end{center}
\caption{\small The PWIT.}
\label{fig:PWIT1}
\end{figure}
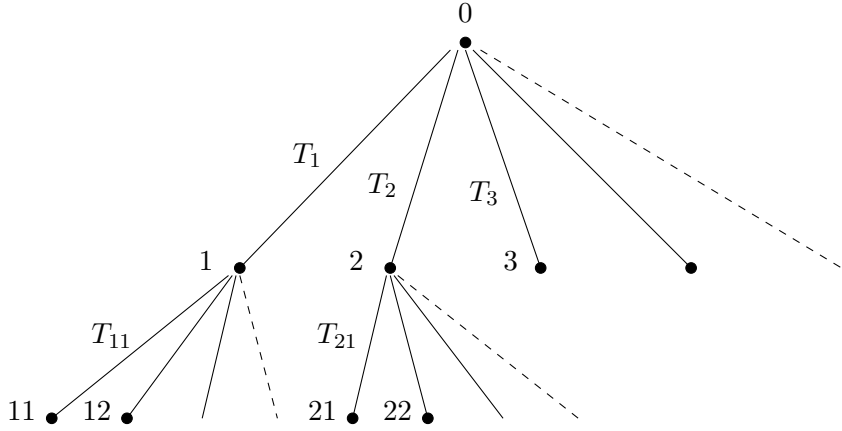

Now consider $K_{n,n}$ with edge costs $\{n\omega(e)\}_{e\in E_n}$, where $\{\omega(e)\}_{e\in E_n}$ are i.i.d.\ exponential with mean 1. With this scaling of the weights, the cheapest edge of a given vertex is Exp(1), the second cheapest is Exp(1)+Exp(1) etc, that is, the ordered weights are described by the arrival times of a rate 1 Poisson process. It is well known that $K_{n,n}$ with costs $\{n\omega(e)\}_{e\in E_n}$ converges to the PWIT in a certain sense. Specifically, write $\mathcal{G}_*$ for the set of rooted weighted graphs satisfying the assumption (ii) of Proposition \ref{prop:existsunique}. It can be shown that $\mathcal{G}_*$ is a complete separable metric space, and a notion of local weak convergence can be defined for probability measures on $\mathcal{G}_*$. A sequence of weighted graphs $\{G_n\}_{n\geq 1}$ converges locally to the PWIT if the following holds: Fix a radius $\rho > 0$ and, given a vertex $v$ of $G_{n}$, consider the subgraph consisting of all paths from $v$ with total cost at most $\rho$. Similarly, consider the subtree of the PWIT consisting of all paths from the root with total cost at most $\rho$. Then, for any given $\rho$, the graph $G_n$ can be coupled with the PWIT so that, with high probability as $n \to \infty$, there is an isomorphism between the two subgraphs which identifies $v$ with the root of the PWIT and which preserves the edge costs. In particular, this means that it is unlikely to encounter short cycles in $G_n$. We refer to \cite{aldous92,aldste04} for further details and a general framework for local weak convergence. Note that also the complete graph $K_n$ with exponential edge weights converges to the PWIT.

\begin{proposition}[Aldous (1992)] \label{prop:convKn}
The complete bipartite graph $K_{n,n}$ with i.i.d.\ exponential edge costs with mean $n$ converges locally to the PWIT:
$$
K_{n,n} \,\, \stackrel{d}{\longrightarrow} \,\,\mathcal{T} \qquad \mbox{as } n \to \infty.
$$
\end{proposition}

Next, we want to apply Proposition \ref{prop:existsunique} to establish the existence of a unique stable matching on the PWIT. To this end, we first recall from \cite[Lemma 4.8]{holmarper20} that the PWIT does not contain infinite descending paths. Here, $|G|$ denotes the number of vertices in a graph $G$.

\begin{proposition}[Holroyd, Peres, Martin (2020)]\label{prop:noinfinitepaths}
Consider $\mathcal{T}$ and its root $0$. For all $s>0$, we have that
\begin{equation}\label{eq:vv}
\E\left[|D_0(\mathcal{T},s)| \right] = e^s.
\end{equation}
 In particular, there are almost surely no infinite descending paths in $\mathcal{T}$.
\end{proposition}

\begin{proof}
For $k \geq 0$, consider descending paths from $0$ of length $k$ and with edge costs less than $s$. Each such path consists of $k$ edges with decreasing costs, where the first edge has cost $s_1<s$, the second edge has cost $s_2 <s_1$, and the $j$th edge has cost $s_j < s_{j-1}$, for $j=3, \dots, k$. The costs along paths of length $k$ can be represented by the points of a unit rate Poisson process on $\RR^k$ and, integrating over the region $0 < s_k < \dots < s_1 < s$, we obtain that the expected number of descending paths of length $k$ with costs less than $s$ is $\int_{0 <s_k < \dots < s_1 < s} \, ds_1 \cdots ds_k = \frac{s^k}{k!}$. Each vertex is the endpoint of at most one such path and thus the expression for $\E\left[|D_0(\mathcal{T},s)| \right]$ follows by summing over $k$.

Recall that $T_n$ is the cost of the edge from the root of $\mathcal{T}$ to its $n$th child. It follows from \eqref{eq:vv} that $|D_0(\mathcal{T},T_n)|$ is finite almost surely for any $n$, implying that $\mathcal{T}$ does not contain infinite descending paths.
\end{proof}

Given this, it is clear that $\mathcal{T}$ satisfies the assumptions of Proposition \ref{prop:existsunique} and we can therefore conclude that it has a unique stable matching.

\begin{proposition}\label{prop:smpwit}
There exists almost surely a unique stable matching $S(\mathcal{T})$ on the PWIT. 
\end{proposition}

Note that we do not yet know that $S(\mathcal{T})$ is perfect. This will follow from Proposition \ref{prop:PWITcost} below. First we note that the set of descending paths in $K_{n,n}$ can be coupled to the set of descending paths in $\mathcal{T}$. This will allow us to derive results for $S(K_{n,n})$ from results for $S(\mathcal{T})$ since, by Proposition \ref{prop:stabledescending}, the stable matching on a graph is determined by descending paths.

\begin{proposition}\label{prop:couplingpaths}
Consider $K_{n,n}$ with exponential edge costs with mean $n$, and fix a vertex $v$. For all $s>0$, there exists a coupling of $D_v(K_{n,n},s)$ and $D_0(\mathcal{T},s)$ such that the weighted graphs coincide with high probability as $n \to \infty$.
\end{proposition}

\begin{proof}
By Proposition \ref{prop:noinfinitepaths}, the set of descending paths $D_0(\mathcal{T},s)$ is contained in the set of paths from $0$ with total weight at most $\rho$ for some value of $\rho<\infty$. The claim hence follows from Proposition \ref{prop:convKn}. 
\end{proof}

Write $W_0$ for the matching cost of the root in the stable matching on the PWIT. We end this section by determining the distribution of $W_0$. Since $W_0$ is finite almost surely, it follows that the stable matching on the PWIT is perfect almost surely. This is proved in \cite[Section 3.2.1]{holmarper20}, but we give a different argument based on Theorem \ref{theorem2} and the connection between the stable matching and descending paths.

\begin{proposition}\label{prop:PWITcost}
We have that $W_0\stackrel{d}{=}W$, where the density of $W$ is given by \eqref{typdens}.
\end{proposition}

\begin{proof} 
Recall that $c(v)$ denotes the matching cost of vertex $v$ in $K_{n,n}$ equipped with i.i.d.\ exponential edge weights with mean 1. Write $\tilde{c}_n(v):=nc(v)$ for the cost when the weights are scaled to have mean $n$. By Theorem \ref{theorem2}, the cost $\tilde{c}_n(v)$ converges in distribution to a proper random variable $W$ with density \eqref{typdens}. The claim hence follows from the uniqueness of the limiting distribution if we show that $\tilde{c}_n(v)$ converges in distribution to $W_0$. To this end, let $\tilde{c}_n^{\sss (s)}(v)$ denote the analogue of $\tilde{c}_n(v)$ in the stable matching on $D_v(K_{n,n},s)$ (based on exponential weights with mean $n$) and, similarly, let $W_0^{\sss (s)}$ be the analogue of $W_0$ on $D_0(\mathcal{T},s)$. By Proposition \ref{prop:stabledescending}, the root is matched to a vertex $u$ in $S(\mathcal{T})$ if and only if it is matched to $u$ in $S(D_0(\mathcal{T},s))$ for large $s$. Furthermore, by Proposition \ref{prop:couplingpaths}, the graphs $D_0(\mathcal{T},s)$ and $D_v(K_{n,n},s)$ can be coupled so that they coincide with high probability as $n\to\infty$. Hence
\begin{equation}\label{eq1}
\PP(W_0>x)=\lim_{s\to\infty}\PP(W_0^{\sss(s)}>x)=\lim_{s\to\infty}\lim_{n\to\infty}\PP(\tilde{c}_n^{\sss(s)}(v)>x).
\end{equation}
If follows from Proposition \ref{prop:stabledescending} applied to $K_{n,n}$ that $\tilde{c}_n^{\sss(s)}(v)\geq \tilde{c}_n(v)$ (with equality if $\tilde{c}_n^{\sss(s)}(v)<\infty$, that is, if $v$ is matched in  $S(D_v(K_{n,n},s))$). Since $\tilde{c}_n(v)$ does not depend on $s$, we obtain that
\begin{equation}\label{eq2}
\lim_{s\to\infty}\lim_{n\to\infty}\PP(\tilde{c}_n^{\sss(s)}(v)>x)\geq \lim_{n\to\infty}\PP(\tilde{c}_n(v)>x).
\end{equation}
To get the reverse inequality, note that, on the event $\{\tilde{c}_n(v)\leq s\}$, we have that $\tilde{c}_n^{\sss(s)}(v)=\tilde{c}_n(v)$, since $v$ is then matched in $S(D_v(K_{n,n},s))$. We can thus bound
$$
 \begin{array}{lll}
\PP(\tilde{c}_n^{\sss(s)}(v)>x) & =&  \PP(\tilde{c}_n^{\sss(s)}(v)>x\cap \tilde{c}_n(v)\leq s) + \PP(\tilde{c}_n^{\sss(s)}(v)>x\cap \tilde{c}_n(v)> s) \\
& \leq & \PP(\tilde{c}_n(v)>x)+\PP(\tilde{c}_n(v)> s).
\end{array}
$$
If follows from Theorem \ref{theorem2} that $\lim_{s\to\infty}\lim_{n\to\infty}\PP(\tilde{c}_n(v)> s)=\lim_{s\to\infty}\PP(W>s)=0$
and hence
\begin{equation}\label{eq3}
\lim_{s\to\infty}\lim_{n\to\infty}\PP(\tilde{c}_n^{\sss(s)}(v)>x)\leq \lim_{n\to\infty}\PP(\tilde{c}_n(v)>x).
\end{equation}
Combining \eqref{eq1}-\eqref{eq3} we conclude that $\tilde{c}_n^{\sss(s)}(v)\stackrel{d}{\to}W_0$, as desired.
\end{proof}

\section{Proofs of Theorems \ref{theorem3} and Theorem \ref{theorem4}}

In this section, we prove Theorem \ref{theorem3} and Theorem \ref{theorem4}. Consider a vertex $v\in V_n$ in $K_{n,n}$ and order the edges emanating from $v$ according to cost, so that $e_1$ is the cheapest edge and $e_n$ the most expensive one. Recall that $R_n$ denotes the rank of the edge used by $v$ in the stable matching, that is, $R_n=m$ if $e_m\in S_{n,n}$. Write $R$ for the analogous quantity on the PWIT:
$$
R=\left\{
        \begin{array}{ll}
        m & \mbox{if $(0,m)\in S(\mathcal{T})$};\\
	\infty & \mbox{if $0$ is not matched in $S(\mathcal{T})$}.
        \end{array}
            \right.
$$
Theorem \ref{theorem3} is a consequence of the following two propositions.

\begin{proposition}\label{prop:rank_conv}
We have that $R_n\stackrel{d}{\to} R$ as $n\to\infty$. 
\end{proposition}

\begin{proposition}\label{prop:rank_PWIT}
The rank $R$ on the PWIT satisfies {\rm{(i)}} and {\rm{(ii)}} of Theorem \ref{theorem3}.
\end{proposition}

\begin{proof}[Proof of Proposition \ref{prop:rank_conv}]
This follows from the same arguments that were used to show that $\tilde{c}_n(v)\to W_0$ in the proof of Proposition \ref{prop:PWITcost}. To see this, first note that scaling the edge costs does not affect the ranking of the edges. We can thus use the scaled edge weights $\{n\omega(e)\}_{e\in E}$, where $\{\omega(e)\}_{e\in E}$ are the original i.i.d.\ edge weights. Let $R_n^{\sss (s)}$ and $R^{\sss (s)}$ denote the analogues of $R_n$ and $R$ in the stable matchings on $D_v(K_{n,n},s)$ and $D_v(\mathcal{T},s)$ respectively, that is, $R_n^{\sss (s)}=m$ if $e_m\in S(D_v(K_{n,n},s))$ and $R^{\sss (s)}=m$ if $e_m\in S(D_v(\mathcal{T},s))$. The proof that $\tilde{c}_n(v)\to W_0$ in the proof of Proposition \ref{prop:PWITcost} can now be applied verbatim with $\tilde{c}_n(v)$ and $W_0$  replaced by $R_n$ and $R$, and with $\tilde{c}_n^{\sss(s)}(v)$ and  $W_0^{\sss(s)}$ replaced by  $R_n^{\sss(s)}$ and $R^{\sss(s)}$.
\end{proof}

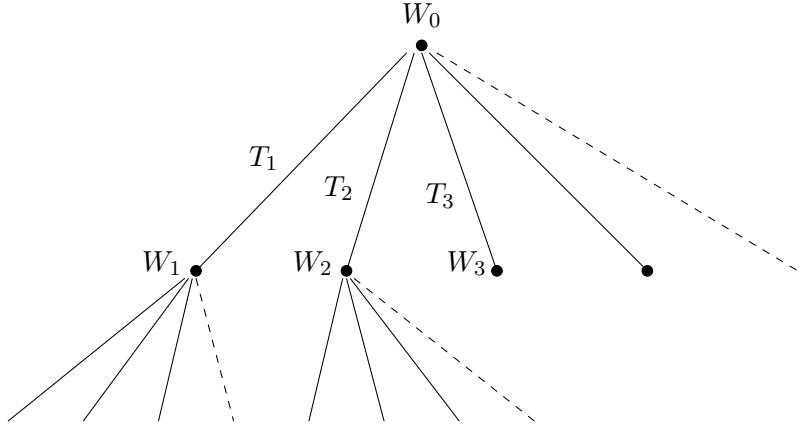
\begin{figure}[htbp]
\begin{center}
\begin{tikzpicture}[scale=1]
\filldraw[black] (0,0) circle (2pt) node{};
\draw (-0.2,-0.1) -- (-3, -3);
\draw (-0.1,-0.1) -- (-1, -3);
\draw (0.0,-0.1) -- (1, -3);
\draw (0.1,-0.1) -- (3, -3);
\draw [dashed] (0.2,-0.1) -- (5, -3);
\filldraw[black] (-3,-3) circle (2pt) node{};
\filldraw[black] (-1,-3) circle (2pt) node{};
\filldraw[black] (1,-3) circle (2pt) node{};
\filldraw[black] (3,-3) circle (2pt) node{};
\draw (-3.15,-3.1) -- (-5.5, -5);
\draw (-3.1,-3.1) -- (-4.5, -5);
\draw (-3.05,-3.1) -- (-3.5, -5);
\draw [dashed] (-3.0,-3.1) -- (-2.5, -5);
\draw (-1.05,-3.1) -- (-1.5, -5);
\draw (-1,-3.1) -- (-0.5, -5);
\draw (-0.95,-3.1) -- (0.5, -5);
\draw [dashed] (-0.9,-3.1) -- (1.5, -5);
\draw (0,0.4) node[] {$W_0$};
\draw (-2.1,-1.5) node[] {$T_1$};
\draw (-1.1,-1.9) node[] {$T_2$};
\draw (0.25,-2) node[] {$T_3$};
\draw (-3.45,-2.9) node[] {$W_1$};
\draw (-1.45,-2.9) node[] {$W_2$};
\draw (0.6,-2.9) node[] {$W_3$};
\end{tikzpicture}
\end{center}
\caption{\small The PWIT with vertices labeled by their matching cost in their respective subgraphs. }
\label{fig:PWIT2}
\end{figure}

\begin{proof}[Proof of Proposition \ref{prop:rank_PWIT}]
For $j=1,2,\ldots$, let $W_j$ denote the matching cost of vertex $j$ in the PWIT in the stable matching on the subgraph consisting of $j$ and its descendants, that is, the edge $(0,j)$ is removed and a stable matching is then constructed on the connected component of vertex $j$; see Figure \ref{fig:PWIT2}. These components have the same structure as the PWIT, implying that $\{W_j\}_{j=0}^\infty$ are i.i.d.\ random variables. By Proposition \ref{prop:PWITcost}, the density is given by \eqref{typdens}. Recall that the cost of the edge $(0,j)$ is $T_j$ and note that $R = \min \{ j \geq 1 :T_j \leq W_j \}$. It follows that
$$
\PP(R=1) =  \PP(U_1 \leq Z_1)= \int_0^{\infty} F_{T_1}(w) f_{W}(w) \, dw= \int_0^{\infty} \frac{1-e^{-w}}{(1+w)^2} \, dw =e \int_1^{\infty} \frac{e^{-t}}{t} \, dt,
$$
where the last integral can be recogniced as -Ei(1) with Ei$(x)= - \int_{-x}^{\infty} \frac{e^{-t}}{t} \, dt $ denoting the exponential integral. This proves (i).

As for (ii), note that $\PP(R > r) = \PP(T_j > W_j, \forall\, j \leq r)$. We can compute this probability by considering an inhomogeneous Poisson process with rate $\lambda(t) = 1-F_W(t)$. Indeed, first consider a standard Poisson process with rate 1 where the event times represent the variables $\{T_j\}_{j \in \mathbb{N}}$, and then generate an inhomogeneous process by accepting an event at time $t$ independently with probability $1-F_W(t)$. The first accepted event is by construction the $R$th event of the original process and $\PP(R >r)$ is then the probability that the inhomogeneous Poisson process has no events before time $T_r$. Hence
$$
\PP(R>r) = \E\left[\PP(\text{no events before } T_r \, |\, T_r)\right] =\E\left[ e^{-\int_0^{T_r} (1-F_W(t)) \, dt} \right] 
= \E \left[\frac{1}{1+ T_r} \right].
$$
Since $\E[T_r] = \Var(T_r) = r$, we have that $T_r \sim r$ as $r\to\infty$ with deviations of order $\sqrt{r}$, and hence $\PP(R \geq r) \sim r^{-1}$.
\end{proof}

It remains to prove Theorem \ref{theorem4}. To this end, a bound on $|D_{v}(K_{n,n},s)|$ uniformly in $n$ is needed. This can be obtained from the bound on $|D_{0}(\mathcal{T},s)|$ in Proposition \ref{prop:noinfinitepaths}.

\begin{lemma}\label{lemmacard}
For $K_{n,n}$ with exponential edge costs with mean $n$, we have that 
$$\PP(|D_{v}(K_{n,n},s)| > e^{2s}) \leq e^{-s}$$ uniformly in $n$. 
\end{lemma}

\begin{proof}
First note that $D_{v}(K_{n,n},s)$ and $D_{v}(K_{n+1,n+1},s)$ can be coupled so that $D_{v}(K_{n,n},s)\subseteq D_{v}(K_{n+1,n+1},s)$. Indeed, if $K_{n+1,n+1}$ is constructed from $K_{n,n}$ by adding one vertex to each of the two vertex sets and equipping the edges of these vertices with i.i.d.\ weights, while the weights of existing edges in $K_{n,n}$ remain the same, then the set of descending paths is non-decreasing. Given this, we obtain that
$$ 
\PP(|D_{v}(K_{n,n},s)| > e^{2s}) \leq \lim_{n \to \infty} \PP(|D_{v}(K_{n,n},s)| > e^{2s}) = \PP(|D_0(\mathcal{T},s)| > e^{2s}) \leq e^{-s},
$$
where the equality follows from Proposition \ref{prop:couplingpaths} and the last inequality follows from \eqref{eq:vv}.
\end{proof}

\begin{proof}[Proof of Theorem \ref{theorem4}]
Recall that $S_{n,n}^\vep$ denotes the stable matching based on edge costs \eqref{eq:omega_t}, that is, a proportion $\vep>0$ of the edge costs $\{\omega(e)\}_{e\in E_n}$ is resampled. Also, for a subgraph $G\subset K_{n,n}$, write $S^\vep(G)$ for the stable matching of $G$ based on the resampled set of edge weights. We will again work with scaled edge costs $\{n\omega_\vep(e)\}_{e\in E_n}$, since this will allow us to make use of Lemma \ref{lemmacard}. Fix a vertex $v\in V_n$ and let $nc(v)$ refer to its matching cost in the initial configuration (with $\vep=0$). We will show that, if $s$ is large, it is unlikely that an edge in $D_v(K_{n,n},s)$ or its boundary is resampled in such a way that the stable matching on $D_v(K_{n,n},s)$ is changed. This will prove the claim since, by Proposition \ref{prop:stabledescending}, vertex $v$ will be matched to the same partner in $S_{n,n}^\vep$ as in $S^\vep(D_v(K_{n,n},s))$ in the limit.

Fix $\vep>0$ and $s>0$, where $s$ will later be chosen as a function of $\vep$. Let $A_{v,n,s}$ be the event that at least one edge in $D_{v}(K_{n,n},s)$ is resampled. By Lemma \ref{lemmacard}, we have that
\begin{equation}
\label{PA}
\begin{split}
\PP(A_{v,n,s}) = &  \,\PP\left(A_{v,n,s} \, | \, |D_{v}(K_{n,n},s)| > e^{2s}\right) \PP\left(|D_{v}(K_{n,n},s)|> e ^{2s}\right) \\
 & +\PP\left(A_{v,n,s}\, | \,  |D_{v}(K_{n,n},s)|\leq e^{2s}\right) \PP\left(|D_{v}(K_{n,n},s)| \leq e^{2s}\right) \\
 \leq &\, e^{-s} + \varepsilon\, e^{2s},
\end{split}
\end{equation}
 uniformly in $n$. Define the edge boundary of $D_{v}(K_{n,n},s)$ to be the set of edges in $K_{n,n}$ with exactly one endpoint in $D_{v}(K_{n,n},s)$. Similarly, let $B_{v,n,s}$ be the event that an edge in the boundary of $D_{v}(K_{n,n},s)$ is resampled and, in addition, that its new (scaled) cost is less than $s$. Using Lemma \ref{lemmacard}, the fact that there are at most $ne^{2s}$ edges on the boundary if $|D_{v}(K_{n,n},s)| \leq e^{2s}$ and a similar split as in \eqref{PA}, we obtain that
\begin{equation}\label{PB}
 \PP(B_{v,n,s}) =  e^{-s} + \varepsilon\, ne^{2s} (1-e^{-s/n}) \leq  e^{-s} + \varepsilon s e^{2s},
\end{equation}
uniformly in $n$. Write $M_{\vep}(v)$ for the matching partner of $v$ in $S_{n,n}^\vep$. If $nc(v)<s$, then, by Proposition \ref{prop:stabledescending}, vertex $v$ is matched in $S^0(D_{v}(K_{n,n},s))$ and the partner is the same as in $S_{n,n}^0$. Furthermore, on the event $A_{v,n,s}^c\cap B^c_{v,n,s}$, the resampled and the non-resampled configurations on $D_{v}(K_{n,n},s))$ coincide and, in addition, no vertex that is matched in $S^\vep(D_{v}(K_{n,n},s)))$ prefers a vertex outside of $D_{v}(K_{n,n},s))$ before its partner in $D_{v}(K_{n,n},s))$. It follows from the same argument as in the proof of Proposition \ref{prop:stabledescending} that $v$ is matched to the same vertex in $S_{n,n}^\vep$ as in $S^\vep(D_{v}(K_{n,n},s)))$. Also, by the above,  $v$ is matched to the same vertex in $S^\vep(D_{v}(K_{n,n},s)))$ as in $S_{n,n}^0$. Hence
\begin{equation*}
\begin{split}
\lim_{n \to \infty} \PP(M_0(v)\neq M_\vep (v)) & \leq \lim_{n \to \infty} \big[\PP(nc(v) \geq s)  + \PP(A_{v,n,s}) +  \PP(B_{v,n,s}) \big]\\
&  \leq \frac{1}{1+s}+ \vep  e^{2s} + \vep s e^{2s}  \\
& \leq \frac{3}{s} + 2\vep s e^{2s},
\end{split}
\end{equation*}
where in the second inequality we have used Theorem \ref{theorem2} and \eqref{PA}--\eqref{PB}, while in the last inequality we have used the fact that $e^{-s} \leq \frac{1}{s}$. Letting $s=C'\log (\frac{1}{\vep})$, with $C'<1/2$, we obtain for some $\delta>0$ that
$$
 2\vep s e^{2s}=2C'\vep^{1-2C'}\log\left(\frac{1}{\vep}\right)\leq \vep^\delta\leq \frac{1}{\log\left(\frac{1}{\vep}\right)}
$$
so that hence
$$
\lim_{n \to \infty} \PP(M_0(v)\neq M_\vep (v))  \leq  \frac{C}{\log\left(\frac{1}{\varepsilon}\right)},
$$
for any $C>7$. Consequently, the probability that the edge $(v,M(v))$ is present both in $S_{n,n}^0$ and in $S_{n,n}^\vep$ is given by 
$$
\lim_{n \to \infty} \PP\left(M_0(v)=M_\vep(v)\right)\geq 1-\frac{C}{\log\left(\frac{1}{\varepsilon}\right)}.
$$
Summing over all $n$ edges in the stable matching gives the desired result
$$ 
\lim_{n \to \infty} \frac{\E\left[|S_{n,n}^0 \cap S_{n,n}^{\varepsilon}| \right]}{n/2} \geq 1 - C \frac{1}{\log\left(\frac{1}{\varepsilon}\right)}.
$$ 
\end{proof}

\section{Sensitivity of the tail of the matching}

In this section, we prove Theorem \ref{thm:tailsensitivity}, stating that the most expensive eges in $S_{n,n}^0$ and $S_{n,n}^\vep$ are with high probability different. Recall that $L_\vep(m)$ denotes the sets of the $m$ most expensive edges in $S_{n,n}^\vep$. To ease notation, we write $L_\vep(m)=L_\vep$ and abbreviate $L_0=L$. Note that, before perturbing the costs, no edge connecting a vertex in $L$ to a vertex in $L^c$ is included in the matching. To establish the theorem, we will show that, for every vertex $u\in L$ there will be many vertices $v\in L^c$ for which the edge cost of $(u,v)$ is resampled in such a way that $v$ prefers to be rematched to $u$. In order to make sure that $u$ also desires $v$, and that the cost of the new match is not among the $m$ most expensive edges in the new matching, we require that resampled edges have costs below a certain threshold $\delta$, which it is unlikely that any edge in $L$ falls below. The core of the proof will be to establish two key lemmas, formalising this outline. First however we will require some information regarding the magnitude and concentration of the edge weights of the stable matching.


\subsection{Concentration of the matching costs}

Recall that $Y_1,Y_2,\ldots,Y_n$ denote the costs of the edges in the stable matching $S_{n,n}^0$, ordered from cheapest to most expensive. By the representation in~\eqref{yk}, the cost of the $k$th cheapest edge is 
\begin{equation}\label{eq:Yk}
Y_k=\sum_{i=1}^kX_i,
\end{equation}
where $X_1,X_2,\ldots,X_n$ are independent and exponentially distributed random variables where the parameter of $X_i$ is $(n-i+1)^2$. It follows, in particular, that
$$
\E[Y_{n-\ell}]=\sum_{i=1}^{n-\ell}\E[X_i]=\sum_{i=\ell+1}^n\frac{1}{i^2}
$$
and
$$
\Var(Y_{n-\ell})=\sum_{i=1}^{n-\ell}\Var(X_i)=\sum_{i=\ell+1}^n\frac{1}{i^4}.
$$
Approximating the sum with an integral leads to
\begin{equation}\label{eq:Y_rev}
\frac{1}{\ell+1}-\frac{1}{n}\le\E[Y_{n-\ell}]\le\frac{1}{\ell}\quad\text{and}\quad\Var(Y_{n-\ell})\le\frac{1}{3\ell^3}.
\end{equation}
This yields the following concentration bounds on the final most expensive edges of the matching.

\begin{lemma}\label{lma:Y_reverse}
For $\ell\ge1$ and $n\ge6\ell$ we have
$$
\PP\Big(Y_{n-\ell}<\frac{1}{6\ell}\Big)\le\frac{12}{\ell}\quad\text{and}\quad\PP\Big(Y_{n-\ell}>\frac{7}{6\ell}\Big)\le\frac{12}{\ell}.
$$
\end{lemma}

\begin{proof}
For $\ell\ge1$ and $n\ge6\ell$ we have from~\eqref{eq:Y_rev} that $\frac{1}{3\ell}\le\E[Y_{n-\ell}]\le\frac{1}{\ell}$, so the result follows from~\eqref{eq:Y_rev} and Chebyshev's inequality.
\end{proof}

Summing over $k$ in~\eqref{eq:Yk} gives the accumulated cost of the edges in the matching. In particular, $C_{n,n}=\sum_{k=1}^nY_k$. The accumulated cost of the first $n-\ell$ edges is
$$
\sum_{k=1}^{n-\ell}Y_k=\sum_{i=1}^{n-\ell}(n-\ell-i+1)X_i.
$$
Hence,
$$
\sum_{k=1}^{n-\ell}\E[Y_k]=\sum_{i=1}^{n-\ell}(n-\ell-i+1)\E[X_i]=\sum_{i=\ell+1}^{n}\frac{i-\ell}{i^2}.
$$
Similarly,
$$
\Var\bigg(\sum_{k=1}^{n-\ell}Y_k\bigg)=\sum_{i=1}^{n-\ell}(n-\ell-i+1)^2\Var(X_i)=\sum_{i=\ell+1}^n\frac{(i-\ell)^2}{i^4}.
$$
Comparing the sums to integrals, for $\ell\ll n$, leads to the bounds
\begin{equation}\label{eq:Y_partsum}
\log\Big(\frac{n}{\ell}\Big)-2\le\sum_{k=1}^{n-\ell}\E[Y_k]\le\log\Big(\frac{n}{\ell}\Big)\quad\text{and}\quad\Var\bigg(\sum_{k=1}^{n-\ell}Y_k\bigg)\sim\frac{1}{3\ell}.
\end{equation}
Finally, reversing the sum and using~\eqref{eq:Y_rev}, we get that for all $n\ge1$
\begin{equation}\label{eq:Y_squares}
\sum_{k=1}^{n}\E[Y_k^2]=\sum_{\ell=1}^n\big(\Var(Y_{n-\ell})+\E[Y_{n-\ell}]^2\big)\le \sum_{\ell=1}^n\Big(\frac{1}{3\ell^3}+\frac{1}{\ell^2}\Big) \leq \frac{\zeta(3)}{3} + \frac{\pi^2}{6} \leq 3,
\end{equation}
where $\zeta(s)$ is the Riemann zeta function.

\subsection{Key lemmas}

Set $\delta:=(\log n)^{-3}$. By Lemma~\ref{lma:Y_reverse}, it is unlikely for the last $m$ edges of the matching to have cost below $\delta$. As we shall see, the threshold $\delta$ is chosen so that it remains unlikely for the cost of edges between vertices in the set $L$ of the $m$ most expensive edges in the original matching to fall below $\delta$ even after the costs have been resampled.

Recall that $\omega_\vep$ denotes the configuration of edge costs after an $\vep$-perturbation. Given a vertex $u$, we denote by $c_\vep^u(v)$ the cost of the vertex $v$ in the stable matching of $K_{n,n}$ with respect to $\omega_\vep$ where $u$ has been removed (and one node is necessarily left unmatched). For $u\in L$ let
\begin{align*}
J_u&:=\big\{v\in L^c:(u,v)\text{ resampled}\big\},\\
N_u&:=\#\big\{v\in J_u:\omega'(u,v)<c_\vep^u(v)\wedge\delta\big\},
\end{align*}
where $\#$ denotes the cardinality of the set. 

The key step towards Theorem~\ref{thm:tailsensitivity} is to show that, with high probability $N_u\ge1$ for all $u\in L$. In order to do that, we compare the costs of the edges in $S_{n,n}$ to the costs of the stable matching of $K_{n,n}$ when a vertex $u$ has been removed. Note that the matching of $K_{n,n}$ with a vertex removed will contain $n-1$ edges, and we denote their weights by $Y_1^u,Y_2^u,\ldots,Y_{n-1}^u$ in increasing order.

\begin{lemma}\label{lma:domination}
Almost surely, we have for every vertex $u$ and all $k=1,2,\ldots,n-1$ that
$$
Y_k\le Y_k^u\le Y_{k+1}.
$$
\end{lemma}

\begin{proof}
Fix a vertex $u$ in $K_{n,n}$. We will take a dynamic perspective on the construction of the matching, where we think of the weights $\omega$ as the times of the first rings of independent Poisson clocks associated with the edges of $K_{n,n}$. We may then construct the matching dynamically in time, by adding an edge $e$ at time $\omega(e)$ unless either of its endpoints has already been matched at an earlier time.

In order to address the discrepancy between $S_{n,n}$ and the matching of $K_{n,n}$ with $u$ removed, with respect to the same configuration $\omega$, we colour `red' the edges in $S_{n,n}$ that are not part of the matching with $u$ removed, and `blue' the edges in the matching with $u$ removed, which are not in $S_{n,n}$. The edges that are in both matchings are not coloured, that is, they remain `black'. Note that the first time a coloured edge is added to either of the two matchings is when $u$ is added to $S_{n,n}$, since this is the only discrepancy between the two weighted graphs. This edge is red, and we denote by $r_1$ its weight, and by $v_1$ its endpoint not equal to $u$. The discrepancy at time $r_1$ is now moved to the vertex $v_1$. Either $v_1$ is left unmatched, or the next coloured edge added to the matching comes when $v_1$ is matched in $K_{n,n}$ with $u$ removed. This edge is thus blue, and we let $b_1$ denote its cost and $u_2$ its endpoint other than $v_1$. Since $v_1$ is added after $u$, we have $r_1<b_1$, and the discrepancy in the two constructions is now transferred to $u_2$. Repeating the above argument we find an alternating sequence of red and blue edges being added to the graph, starting and ending with a red edge, whose weights are similarly alternating
$$
r_1<b_1<r_2<b_2\ldots<r_\ell
$$
for some $\ell\ge1$. Since coloured edges are added alternatingly, there is at any time at most one more red than blue edge present, and never more blue than red. Consider the $k$th edge added to the matching of $K_{n,n}$ with $u$ removed, which happens at time $Y_k^u$. Either this edge is black, in which case it is either the $k$th or $(k+1)$st edge added to $S_{n,n}$, and hence $Y_k^u$ equals either $Y_k$ or $Y_{k+1}$. Or the edge is blue, in which case $S_{n,n}$ already consists of $k$ but not $k+1$ edges, and so $Y_{k}<Y_k^u<Y_{k+1}$. This holds for every $u$ and $k=1,2,\ldots,n-1$.
\end{proof}

Our main step towards Theorem~\ref{thm:tailsensitivity} is a moment analysis of $N_u$ for $u\in L$. For ease of notation, we shall let $\PP':=\PP(\,\cdot\,|(L,L^c))$, and write $\E'$ and $\Var'$ for expectation and variance with respect to $\PP'$. Note that the law of the cost of a vertex, under $\PP'$, depends on whether the vertex belongs to $L$ or $L^c$, whereas the law of $(Y_1,Y_2,\ldots,Y_n)$ is equal under $\PP$ and $\PP'$. 

\begin{lemma}\label{lma:Nu}
For $m\ge1$ and $\vep\in(0,1]$ satisfying $2m\le\vep\log n$, we have that for every vertex $u\in L$ the two following statements hold:
\begin{itemize}
\item[\rm{(i)}] $\E'[N_u]=(1+o(1))\vep\log n$;
\item[\rm{(ii)}] $\Var'(N_u)\le6\vep\log n$.
\end{itemize}
\end{lemma}

\begin{proof}
We prove the two statements separately.
\begin{itemize}
\item[(i)]
Fix a vertex $u\in L$ and let $F(x)=1-e^{-x}$ denote the distribution function of the exponential distribution. By conditioning on everything but the update variables $U_e$ for the edges that connect $u$ to $L^c$, we find that
$$
\E'[N_u]=\vep\sum_{v\in B}\PP'\big(\omega'(u,v)<c_\vep^u(v)\wedge\delta\big).
$$
Conditioning, this time on the weight configuration $\omega_\vep$ for edges not incident to $u$, we find that
$$
\E'[N_u]=\vep\sum_{v\in B}\E'\big[F(c_\vep^u(v)\wedge\delta)\big]
$$
Define $H:=\sum_{v\in L^c}c_\vep^u(v)\wedge\delta$. Using that $x-\frac12x^2\le F(x)\le x$, Lemma~\ref{lma:domination} and~\eqref{eq:Y_squares}, we obtain that
\begin{equation}\label{eq:ENu}
\Big|\E'[N_u]-\vep\, \E'[H]\Big|\le\frac\vep2\sum_{v\in L^c}\E'\big[(c_\vep^u(v)\wedge\delta)^2\big]\le\frac{\vep}{2}\sum_{k=1}^n\E[Y_k^2]\le2\vep.
\end{equation}
Another application of Lemma~\ref{lma:domination} gives that
\begin{equation} \label{eq:upperbounds}
\sum_{k=1}^{n-m}\E[Y_k\wedge\delta]\le\E'[H]\le\sum_{k=1}^n\E[Y_k].
\end{equation}
Let $\ell_n=\lceil2(\log n)^3\rceil$ and set $E=\{Y_{n-\ell_n}\le\delta\}$, where $\delta = (\log n)^{-3}$. Then, Lemma~\ref{lma:Y_reverse} gives $\PP(E^c)\le6/(\log n)^3$. Consequently, using Cauchy-Schwartz' inequality and Theorem~\ref{theorem1}, we obtain that
$$
\E\bigg[\sum_{k=1}^nY_k{\bf 1}_{E^c}\bigg]\le \sqrt{\E[C_{n,n}^2]\PP(E^c)}\le\sqrt{\big(\Var(C_{n,n})+\E[C_{n,n}]^2\big)\PP(E^c)}\le\frac{4}{\sqrt{\log n}}.
$$
Hence, for large $n$,
$$
\sum_{k=1}^{n-m}\E[Y_k\wedge\delta]\ge\E\bigg[\sum_{k=1}^{n-\ell_n}Y_k{\bf 1}_{E}\bigg]\ge\E\bigg[\sum_{k=1}^{n-\ell_n}Y_k\bigg]-\E\bigg[\sum_{k=1}^{n-\ell_n}Y_k{\bf 1}_{E^c}\bigg]\ge\sum_{k=1}^{n-\ell_n}\E[Y_k]-1,
$$
which, combined with~\eqref{eq:Y_partsum} and~\eqref{eq:upperbounds}, gives
\begin{equation}\label{eq:cb_bound}
\log n-4\log\log n\le\sum_{k=1}^{n-\ell_n}\E[Y_k]-1\le\E'[H]\le\sum_{k=1}^n\E[Y_k]\le\log n.
\end{equation}
Together with~\eqref{eq:ENu}, this shows that $\E'[N_u]=(1+o(1))\vep\log n$.

\item[(ii)]
First note that we can express $N_u$ as
$$
N_u=\sum_{v\in L^c}{\bf 1}_{\{v\in J_u\}}{\bf 1}_{\{\omega'(u,v)<c_\vep^u(v)\wedge\delta\}}.
$$
Expanding the square and conditioning, first on everything but the update variables $U_e$ for edges that connect $u$ to $L^c$, and then on the weight configuration $\omega_\vep$ for edges not incident to $u$, we obtain that
\begin{equation*}
\begin{aligned}
\E'[N_u^2]&=\vep\sum_{v\in L^c}\E'\big[F(c_\vep^u(v)\wedge\delta)\big]+\vep^2\sum_{\substack{v,v' \in L^c \\ v\neq v'}}\E'\big[F(c_\vep^u(v)\wedge\delta)F(c_\vep^u(v')\wedge\delta)\big]\\
&\le\E'[N_u]+\vep^2\,\E'\bigg[\Big(\sum_{v\in L^c}F(c_\vep^u(v)\wedge\delta)\Big)^2\bigg]\\
&\le\E'[N_u]+\vep^2\,\E'\bigg[\Big(\sum_{v\in L^c}c_\vep^u(v)\wedge\delta\Big)^2\bigg] \\
& \le \E'[N_u]+\vep^2\,\E'\left[ H^2 \right].
\end{aligned}
\end{equation*}
Combining the above with~\eqref{eq:ENu} and~\eqref{eq:cb_bound}, we get that
\begin{equation}\label{eq:var_bound}
\Var'\big(N_u\big) = \E'[N_u^2] - \E'[N_u]^2 \le \E'[N_u]+\vep^2\bigg(\Var'(H)+4\log n\bigg).
\end{equation}
Next, we introduce three events $D_1=\{Z>b\}$, $D_2=\{Z\in[a,b]\}$ and $D_3=\{Z<a\}$, where $a=\sum_{k=1}^{n-\ell_n}\E[Y_k]-1$ and $b=\sum_{k=1}^n\E[Y_k]$. We bound the variance of $H$ by estimating its contribution restricted to each of the events $D_1$, $D_2$ and $D_3$, that is,
$$
\Var'(H)=\E'\big[(H-\E'[H])^2{\bf 1}_{D_1}\big]+\E'\big[(H-\E'[H])^2{\bf 1}_{D_2}\big]+\E'\big[H-\E'[H])^2{\bf 1}_{D_3}\big].
$$
First, we note from~\eqref{eq:cb_bound} that $a\le\E'[H]\le b$ and that $b-a\le 4\log\log n$, which immediately gives
$$
\E'\big[(H-\E'[H])^2{\bf 1}_{D_2}\big]\le(b-a)^2\le16(\log\log n)^2.
$$
Second, by adding and subtracting $b$, and using that $(x+y)^2\le 4x^2+4y^2$, we find that
$$
\E'\big[(H-\E'[H])^2{\bf 1}_{D_1}\big]\le 4\,\E'\big[(H-b)^2{\bf 1}_{D_1}\big]+4\,\big(\E'[H]-b\big)^2.
$$
Using Lemma~\ref{lma:domination}, and that we have restricted to the event $D_1$, we obtain the further upper bound
$$
4\,\E'\bigg[\Big(\sum_{k=1}^nY_k-b\Big)^2{\bf 1}_{D_1}\bigg]+4(b-a)^2\le 4\,\Var(C_{n,n})+64(\log\log n)^2.
$$
Third, by adding and subtracting $a$, and using that $(x+y)^2\le 4x^2+4y^2$, we find that
$$
\E'\big[(H-\E'[H])^2{\bf 1}_{D_3}\big]\le4\,\E'\big[(H-a)^2{\bf 1}_{D_3}\big]+4\,\big(\E'[H]-a\big)^2.
$$
Recall that $E=\{Y_{n-\ell_n}\le\delta\}$. Using Lemma~\ref{lma:domination} and~\eqref{eq:Y_partsum}, and that we have restricted to the event $D_3$, we obtain that
\begin{align*}
\E'\big[(H-a)^2{\bf 1}_{D_3}\big]&\le\E\bigg[\Big(\sum_{k=1}^{n-\ell_n}Y_k\wedge\delta-a\Big)^2{\bf 1}_{D_3}\bigg]\le\E\bigg[\Big(\sum_{k=1}^{n-\ell_n}Y_k-a\Big)^2{\bf 1}_{D_3\cap E}\bigg]+\E\big[a^2{\bf 1}_{D_3\cap E^c}\big]\\
&\le \Var\bigg(\sum_{k=1}^{n-\ell_n}Y_k\bigg)+a^2\,\PP(E^c)=o(1).
\end{align*}
Hence, for large $n$,
$$
\E'\big[(H-\E'[H])^2{\bf 1}_{D_3}\big]\le 1+4(b-a)^2\le 1+64(\log\log n)^2.
$$
In conclusion, we get that $\Var'(H)\le 200(\log\log n)^2$ and hence, via~\eqref{eq:var_bound} that for $n$ large enough 
$$\Var'(N_u)\le \vep \log n +5\vep^2\log n \leq 6\vep \log n.$$
\end{itemize}
\end{proof}

\subsection{Proof of Theorem~\ref{thm:tailsensitivity}}

Recall that for $\vep \in (0,1]$, we assume that $m \ll \vep \log n$ and $\delta = (\log n)^{-3}$. Let $G_1=\{Y_{n-m}>\delta,Y_{n-m}^\vep>\delta\}$, where $Y_k^\vep$ denotes the $k$th cheapest edge in the matching $S^{\vep}_{n,n}$, $G_2=\{\omega'(u,u')>\delta\text{ for all }u,u'\in L\}$, and $G_3=\bigcap_{u\in L}\{N_u\ge1\}$. Finally, set $G=G_1\cap G_2\cap G_3$.

We start by bounding the probability that $G$ fails. First, by Lemma~\ref{lma:Y_reverse} we have
$$
\PP(G_1^c)\le2\PP(Y_{n-m}\le\delta)\le2\PP(Y_{n-\lceil\log n\rceil}\le\delta)\le\frac{24}{\log n}.
$$
Second, conditioning on the division $(L,L^c)$ and using the union bound, we have that
$$
\PP(G_2^c)\le\E\bigg[\sum_{u,u'\in L}\PP'\big(\omega'(u,u')\le\delta\big)\bigg]\le m^2 F(\delta) \le m^2 \frac{1}{(\log n)^3},
$$
where $F(x)=1-e^{-x}\le x$ again denotes the distribution function of the exponential distribution. Third, the union bound, Chebyshev's inequality and Lemma~\ref{lma:Nu} give that
$$
\PP(G_3^c)\le \E\bigg[\sum_{u\in L}\PP'(N_u=0)\bigg]\le \E\bigg[\sum_{u\in L}4\frac{\Var'(N_u)}{\E'[N_u]^2}\bigg]\le 2m\frac{24}{\vep\log n}.
$$
In conclusion, 
\begin{equation}\label{eq:tailquant}
\PP(G^c) \leq \PP(G_1^c) + \PP(G_2^c) +\PP(G_3^c) \leq \frac{24}{\log n} + \frac{m^2}{(\log n)^3} + \frac{48m}{\vep \log n},
\end{equation}
which is $o(1)$ since $m \ll \vep \log n$. Hence $G$ occurs with high probability as $n\to\infty$. 

Note that, on $G_1$, we have $\omega(e)>\delta$ for every $e\in L$ and $\omega_\vep(e)>\delta$ for every $e\in L_\vep$. Moreover, on $G_1\cap G_2$ we have $\omega_\vep(e)>\delta$ for every $e\in L$, that is, all edges originally in $L$ still have cost exceeding $\delta$ after perturbation. We claim that, on $G_3$, every node in $L$ is matched with cost at most $\delta$ after perturbation. Once the claim is proved, we conclude that, on the event $G$, we have that after perturbation:
\begin{itemize}
\item every node in $L$ is rematched with cost at most $\delta$;
\item every edge in $L$ has cost exceeding $\delta$ and hence does not belong to the matching $S^{\vep}_{n,n}$;
\item the $m$ most expensive edges have cost exceeding $\delta$, implying that $L_\vep\subseteq L^c$.
\end{itemize}

It remains to prove the claim that, on $G_3$, every node in $L$ is matched with cost at most $\delta$ after perturbation. We again argue using a dynamic construction of the matching in which an edge is added to the matching at the `time' indicated by its cost, unless either of its endpoints has already been matched before that time. It is straightforward to verify that the matching obtained is indeed the stable matching $S^{\vep}_{n,n}$. In the dynamic construction, a vertex being unmatched at time $\delta$ is equivalent to the cost of the vertex exceeding $\delta$. Consequently, if a vertex $u\in L$ is left unmatched at time $\delta$ in the perturbed configuration, we have $c_\vep(u)>\delta$. Assume that a vertex $u$ is unmatched at time $\delta$, which implies that the matching obtained at time $\delta$ coincides with the matching obtained until time $\delta$ when $u$ is removed. In particular, it follows that
\begin{equation}\label{eq:u_removal}
c_\vep(v)\wedge\delta=c_\vep^u(v)\quad\text{for every }v\in L^c.
\end{equation}
On $G_3$, we have that $N_u\ge1$ for every $u\in L$. Hence there exists a vertex $v\in L^c$ such that
$$
\omega_\vep(u,v)<c_\vep^u(v)\wedge\delta.
$$
This contradicts~\eqref{eq:u_removal}, since it implies the existence of a vertex $v\in L^c$ which is unmatched at time $\omega_\vep(u,v)<\delta$, to which $u$ would therefore be matched to, unless it has already been matched before. In conclusion, on $G_3$ every node in $L$ is matched with cost at most $\delta$ after perturbation, as required. This ends the proof of the theorem. \qed

\section{Noise sensitivity of the stable matching}

In this section, we prove Theorem \ref{thm:weightsensitivity}. The proof will roughly go as follows. We first observe that the bulk of the matching is responsible for most of the matching cost, whereas the cost of the bulk of the bulk of the matching is highly concentrated, so that most of the randomness comes from the tail of the last edges. We then dynamically construct the matchings $S_{n,n}$ and $S_{n,n}^{\vep}$, by equipping each edge with a Poisson clock and adding it to the corresponding matching when its clock rings, if adding the edge is allowed. By concentration of the bulk of the matching, most edges added are the same in both matchings. However, by Theorem~\ref{thm:tailsensitivity}, we will reach a point in time when the  remaining sets of unmatched vertices correspond to disjoint subgraphs. From this point on, we are waiting for independent sets of clocks to ring. The contributions to the matchings obtained from this phase will therefore be independent and, since this phase is responsible for most of the randomness in the construction of the matching, the correlation of the matching costs $C_{n,n}^0$ and $C_{n,n}^\vep$ will be small.

Given $m\ge1$, denote by $W_m^-(\omega)$ and $W_m^+(\omega)$ the cost of the matching that is detected in the matching of the first $n-m$ and last $m$ edges, respectively. In the notation of~\eqref{yk}, we have
$$
W_m^-(\omega)=\sum_{k=1}^{n-m}(n-k+1)X_k\quad\text{and}\quad W_m^+(\omega)=\sum_{k=n-m+1}^n(n-k+1)X_k.
$$
Note that $C_{n,n}=W_m^-+W_m^+$. In particular we find that
$$
\E[W_m^-]=\sum_{k=m+1}^n\frac{1}{k}\quad\text{and}\quad\E[W_m^+]=\sum_{k=1}^m\frac{1}{k},
$$
and hence that $\E[W_m^-]\sim\log(n/m)$ and $\E[W_m^+]\sim\log m$ for $1\ll m\ll n$. In addition, we have
$$
\Var(W_m^-)=\sum_{k=m+1}^n\frac{1}{k^2}\le\frac{1}{m}\quad\text{and}\quad\Var(W_m^+)=\sum_{k=1}^m\frac{1}{k^2}\le\frac{\pi^2}{6}.
$$
That is, while little weight remains to be picked up at the end of the matching, most of the randomness comes from that part.

\begin{proof}[Proof of Theorem~\ref{thm:weightsensitivity}]
Fix $m\ge1$. We first decompose the covariance according to
\begin{equation}\label{eq:covariance}
\begin{aligned}
\Cov\big(C_{n,n}^0,C_{n,n}^\vep)\big)&=\Cov\big(W_m^-(\omega),W_m^-(\omega_\vep)\big)+\Cov\big(W_m^-(\omega),W_m^+(\omega_\vep)\big)\\
&\quad+\Cov\big(W_m^+(\omega),W_m^-(\omega_\vep)\big)+\Cov\big(W_m^+(\omega),W_m^+(\omega_\vep)\big).
\end{aligned}
\end{equation}
Since $\Var(W_m^-)\le\frac{1}{m}$, an application of Cauchy-Schwartz gives
$$
\big|\Cov\big(W_m^-(\omega),W_m^-(\omega_\vep)\big)\big|\le\Var(W_m^-)\le\frac{1}{m},
$$
and
$$
\big|\Cov\big(W_m^-(\omega),W_m^+(\omega_\vep)\big)\big|\le\sqrt{\Var(W_m^-)\Var(W_m^+)}\le\frac{2}{\sqrt{m}}.
$$
This gives
\begin{equation}\label{eq:cov_bound}
\big|\Cov\big(C_{n,n}^0,C_{n,n}^\vep\big)\big|\le\big|\Cov\big(W_m^+(\omega),W_m^+(\omega_\vep)\big)\big|+\frac{5}{\sqrt{m}}.
\end{equation}

Write $T$ for the time at which the two subgraphs induced by the unmatched nodes in the two configurations $\omega$ and $\omega_\vep$ become disjoint. Let
$$
Q:=\big\{T\le\min\{Y_{n-m},Y_{n-m}^\vep\}\big\}
$$
and note that, on the event $Q$, when matching the last $m$ edges, we are waiting for disjoint sets of Poisson clocks to ring. We next show that $Q$ occurs with high probability. Set $\ell=\sqrt{\vep\log n}$ and let $Q_1$ denote the event that the subgraphs induced by the vertices of the last $7\ell$ edges of the matching in $\omega$ and $\omega_\vep$ are disjoint. In addition, let
\begin{align*}
Q_2&=\Big\{\max\{Y_{n-7\ell},Y_{n-7\ell}^\vep\}\le\frac{1}{6\ell}\Big\},\\
Q_3&=\Big\{\min\{Y_{n-\ell},Y_{n-\ell}^\vep\}\ge\frac{1}{6\ell}\Big\}.
\end{align*}
By assumption, we have $\vep\log n\gg1$, so that $m\le\ell$ when $n$ is large. It follows that, on $Q_1\cap Q_2\cap Q_3$, we have for large $n$ that
$$
T\le\max\{Y_{n-7\ell},Y_{n-7\ell}^\vep\}\le\min\{Y_{n-\ell},Y_{n-\ell}^\vep\}\le\min\{Y_{n-m},Y_{n-m}^\vep\},
$$fd
so that $Q_1\cap Q_2\cap Q_3\subseteq Q$. Note that the probability of $Q_1^c$ can be upper bounded using the quantitative bound \eqref{eq:tailquant} leading to Theorem~\ref{thm:tailsensitivity}, while the probabilities of $Q_2^c$ and $Q_3^c$ can be bounded using Lemma~\ref{lma:Y_reverse}. Hence
\begin{equation}
\label{eq:Gcompl}
\PP(Q^c)\le\PP(Q_1^c)+\PP(Q_2^c)+\PP(Q_3^c)\le\frac{350}{\sqrt{\vep\log n}}+2\frac{12}{7\ell}+2\frac{12}{\ell} \le\frac{400}{\sqrt{\vep\log n}}
\end{equation}
for sufficiently large $n$.

Decomposing the covariance depending on the event $Q$ gives
\begin{equation}\label{eq:cov_decomp}
\begin{aligned}
\Cov\big(W_m^+(\omega),W_m^+(\omega_\vep)\big)&=\E\big[\big(W_m^+(\omega)-\E[W_m^+]\big)\big(W_m^+(\omega_\vep)-\E[W_m^+]\big){\bf 1}_{Q}\big]\\
&\quad+\E\big[\big(W_m^+(\omega)-\E[W_m^+]\big)\big(W_m^+(\omega_\vep)-\E[W_m^+]\big){\bf 1}_{Q^c}\big].
\end{aligned}
\end{equation}
Note that $Q$ depends on the Poisson clocks in $\omega$ and $\omega_\vep$ before time $\min\{Y_{n-m},Y_{n-m}^\vep\}$, whereas $W_m^+(\omega)$ and $W_m^+(\omega_\vep)$ depend on the clocks after time $\min\{Y_{n-m},Y_{n-m}^\vep\}$. Moreover, on $Q$, we have that $W_m^+(\omega)$ and $W_m^+(\omega_\vep)$ are functions of disjoint sets of clocks. It follows that, on $Q$, we have
$$
\big(W_m^+(\omega),W_m^+(\omega_\vep)\big)\stackrel{d}{=}\big(W_m^+(\omega),W_m^+(\tilde\omega)\big),
$$
and hence that
$$
\E\big[\big(W_m^+(\omega)-\E[W_m^+]\big)\big(W_m^+(\omega_\vep)-\E[W_m^+]\big){\bf 1}_{Q}\big]=\E\big[\big(W_m^+(\omega)-\E[W_m^+]\big)\big(W_m^+(\tilde\omega)-\E[W_m^+]\big){\bf 1}_{Q}\big],
$$
where $\tilde \omega$ indicates a cost configuration independent from $\omega$. Using that $xy\le x^2+y^2$, on $Q^c$ we get
$$
\E\big[\big(W_m^+(\omega)-\E[W_m^+]\big)\big(W_m^+(\omega_\vep)-\E[W_m^+]\big){\bf 1}_{Q^c}\big]\le2\,\E\big[\big(W_m^+(\omega)-\E[W_m^+]\big)^2{\bf 1}_{Q^c}\big]
$$
and similarly
$$
\E\big[\big(W_m^+(\omega)-\E[W_m^+]\big)\big(W_m^+(\tilde\omega)-\E[W_m^+]\big){\bf 1}_{Q^c}\big]\le2\,\E\big[\big(W_m^+(\omega)-\E[W_m^+]\big)^2{\bf 1}_{Q^c}\big].
$$
From~\eqref{eq:cov_decomp} we obtain
$$
\big|\Cov\big(W_m^+(\omega),W_m^+(\omega_\vep)\big)\big|\le\big|\Cov\big(W_m^+(\omega),W_m^+(\tilde\omega)\big)\big|+4\,\E\big[\big(W_m^+(\omega)-\E[W_m^+]\big)^2{\bf 1}_{Q^c}\big].
$$
Note that, for fixed $m$ and for $\vep=\vep(n)$ such that $\vep\log n\to\infty$ as $n\to\infty$, we have from~\eqref{eq:Gcompl} that $\PP(Q^c) = o(1)$. Moreover, $W_m^+(\omega)$ and $W_m^+(\tilde\omega)$ are independent and $W_m^+(\omega)$ equals $C_{m,m}$ in distribution. It hence follows from dominated convergence (or the reverse Fatou lemma) that
$$
\Cov\big(W_m^+(\omega),W_m^+(\omega_\vep)\big)\to 0.
$$
Since $m$ was arbitrary, we conclude from~\eqref{eq:cov_bound} that
$$
\Cov\big(C_{n,n}^0,C_{n,n}^\vep\big)\to0.
$$
This completes the proof, since $\Var(C_{n,n})\to\pi^2/6$ as $n\to\infty$.
\end{proof}

\subsection*{Acknowledgement.} The authors thank Svante Janson for discussions during the early part of this project, and in particular for indicating the correct limit law for the matching cost. The first author further thanks Marcelo Campos, Simon Griffiths and Rob Morris for discussions regarding the possible sensitivity of the tail of the matching. This work was in part supported by the Swedish Research Council (VR) through grant 2021-03964 (DA) and 2020-04479 (MD and MS).

\newcommand{\noopsort}[1]{}\def\cprime{$'$}

\end{document}